\documentclass[a4paper]{article}
\usepackage[all]{xy}\usepackage[latin1]{inputenc}        
\usepackage[dvips]{graphics,graphicx}
\usepackage{amsfonts,amssymb,amsmath,color,mathrsfs, amstext}
\usepackage{amsbsy, amsopn, amscd, amsxtra, amsthm,authblk}
\usepackage{enumerate,algorithmic,algorithm}
\usepackage{upref}
\usepackage{geometry}
\usepackage[displaymath]{lineno}
\usepackage{float}
\usepackage{yhmath}
\usepackage{booktabs}
\usepackage{subcaption}
\usepackage{multirow}
\usepackage{makecell}
\usepackage[colorlinks,
            linkcolor=red,
            anchorcolor=red,
            citecolor=red
            ]{hyperref}
\usepackage{upgreek}

\numberwithin{equation}{section}

\def\cL{\mathcal{L}}
\def\cP{\mathcal{P}}
\def\cH{\mathcal{H}}
\def\cX{\mathcal{X}}
\def\cY{\mathcal{Y}}

\def\R{\mathbb{R}}

\def\T{\mathbb{T}}

\newtheorem{theorem}{Theorem}[section]
\newtheorem{lemma}{Lemma}[section]
\newtheorem{proposition}{Proposition}[section]
\newtheorem{remark}{Remark}[section]
\newtheorem{corollary}{Corollary}[section]
\newtheorem{definition}{Definition}[section]
\newtheorem{assumption}{Assumption}[section]

\begin{document}

\title{Accelerating sampling via asymptotic relaxation enhancing flows}

\author[a]{Yuanyuan Feng\thanks{
E-mail: yyfeng@math.ecnu.edu.cn}}
\author[b]{Lei Li\thanks{E-mail: leili2010@sjtu.edu.cn}}
\author[c]{Jian-Guo Liu\thanks{E-mail: jliu@math.duke.edu}}
\author[d]{Xiaoqian Xu\thanks{E-mail: xiaoqian.xu@dukekunshan.edu.cn}}

\affil[a]{School of Mathematical Sciences,  Key Laboratory of MEA (Ministry of Education) \& Shanghai Key Laboratory of PMMP,  East China Normal University, Shanghai 200241, China}
\affil[b]{School of Mathematical Sciences, Institute of Natural Sciences, MOE-LSC, Shanghai Jiao Tong University, Shanghai, 200240, P.R.China.}
\affil[c]{Department of Mathematics, Department of Physics, Duke University, Durham, NC 27708, USA.}
\affil[d]{Zu Chongzhi Center, Duke Kunshan University, Kunshan, 215316, P.R.China}

\date{}

\maketitle

\begin{abstract}
In this paper, we accelerate Langevin Monte Carlo sampling from Gibbs measures $\pi\propto \exp(-U)$ by adding a large drift that preserves the invariant measure.
For warm-start initial data, we characterize the sharp asymptotic decay rate of the relative entropy and introduce asymptotic relaxation enhancing flows: sequences that achieve arbitrarily fast decay.
We construct such flows on the torus by scaling cellular flows and pushing them forward via diffeomorphisms, and we extend the construction to the full space using a Lyapunov function method to control behavior at infinity without periodization, obtaining explicit finite energy flows that guarantee arbitrarily fast convergence under natural growth conditions on $U$.
\end{abstract}

\section{Introduction}\label{section:introduction}

Sampling from a probability distribution of the form  
$$
\pi(x)\propto e^{-U(x)},\qquad x\in\Omega,
$$
is a fundamental task in computational science, where $\Omega$ is the spatial domain. Note that we have set the inverse temperature to $1$ for simplicity, and the general value of inverse temperature follows by redefinition of $U$ or possible time-scaling. The applications of sampling range from molecular dynamics to Bayesian inference \cite{andrieu2003introduction,gelman1995bayesian}. We will consider $\Omega=\mathbb{T}^d$ or $\R^d$ in this work for convenience. A widely used method to sample from this probability measure is the overdamped Langevin Monte Carlo (LMC), which simulates  
\begin{equation}\label{eq:langevin}
dX_t = -\nabla U(X_t)\,dt + \sqrt{2}\,dW_t,
\end{equation}
whose invariant measure is precisely $\pi$. 

LMC is popular because it requires only gradient information of $U$ and can be implemented easily even in high dimensions \cite{andrieu2003introduction, gelman1995bayesian}, but a key challenge is to quantify how fast the law of the discretized process approaches $\pi$. Traditional weak convergence analysis gives error bounds that depend on a fixed test function, and do not directly control distances between probability measures \cite{milstein2004stochastic}. A more powerful approach is to use the \textbf{relative entropy} (Kullback-Leibler divergence) (see \eqref{eq:relativeentropy} below).
The relative entropy is non-negative, convex, and, more importantly, it controls both the total variation (via Pinsker's inequality) and the Wasserstein distances (via Talagrand-type transport inequalities) \cite{pinsker1964information, talagrand1996transportation, otto2000generalization}. Moreover, for Langevin dynamics, the time derivative of the relative entropy with respect to the invariant measure has a simple dissipative form, which makes it a natural tool for convergence analysis \cite{villani2009hypocoercivity}. This idea has been used to obtain sharp convergence rates for numerical schemes. For instance, sharp $O(\Delta t^2)$ bounds ($\Delta t$ is the time step) for the relative entropy error of the Euler-Maruyama scheme have been established for additive noise \cite{mou2022improved, li2025sharp,li2025relative}, and have recently been extended to multiplicative noise using Malliavin calculus \cite{li2024estimates}. Moreover, for the random splitting Langevin Monte Carlo, fourth-order relative entropy bounds have been established, leading to second-order Wasserstein-1 error \cite{li2025ergodicity}. These developments show that relative entropy provides rigorous and sharp bounds  for sampling error. The benefit of the relative entropy lies in the fact that the derivative can be computed conveniently and the so-called logarithmic Sobolev inequality could be utilized to close the estimate yielding exponential ergodicity and sharp error bounds of the Langevin dynamics under mild conditions on $U$.

Though the log-Sobolev inequality gives exponential ergodicity, the convergence of standard LMC can be extremely slow (i.e., the rate in the exponential ergodicity is very small) when the potential is non-convex, due to the metastability phenomenon \cite{freidlin1998random, bovier2004metastability}. To overcome this, one can add a drift that preserves the invariant measure but actively mixes the distribution. Consider the modified dynamics  
\begin{equation}\label{eq:langevinFlow}
dX_t = Av(X_t)\,dt - \nabla U(X_t)\,dt + \sqrt{2}\,dW_t,
\end{equation}
where $A\gg 1$ and $\nabla\cdot(\pi v)=0$. When the drift is chosen appropriately, it can dramatically accelerate convergence. A simple choice is to take $v = J\nabla U$ with an antisymmetric matrix $J$, which gives a constant-factor speed-up but does not eliminate the exponential dependence on the temperature \cite{hwang1993accelerating,lelievre2013optimal,rey2015irreversible,duncannonreversible,duncan2016variance,lu2018analysis}. A more powerful class of drifts is given by \textbf{relaxation enhancing} flows, introduced by Constantin et al. \cite{CKRZ08}. A divergence-free flow $u$ on a compact manifold is called relaxation enhancing if $u\cdot\nabla$ has no non-constant $H^1$ eigenfunctions; such flows make the $L^2$ norm of solutions to the advection-diffusion equation arbitrarily small in arbitrarily short time when the flow amplitude is large. Quantitative refinements of this phenomenon have been developed in subsequent works, which relate dissipation time and energy decay explicitly to the underlying mixing rate~\cite{FI19, CZDG20}.
The acceleration provided by such flows can be quantified in two complementary ways: the spectral gap asymptotics of Franke et al. \cite{franke2010behavior} and the resolvent estimate of Wei \cite{Wei18Mix}, the latter linking the semigroup decay to $\Psi_A=\inf|(\mathcal{L}_A-i\lambda)f|$. All of the above results concern compact manifolds (e.g., the torus $\mathbb{T}^d$). The situation on unbounded domains is more subtle. In particular, Zlato\v{s} \cite{zlatovs2010diffusion} characterized dissipation-enhancing flows on $\mathbb{R}^2$ and $\mathbb{R}\times\mathbb{T}$, and applied them to quenching of reactions. The same principle, which is using relaxation enhancing flows or mixing flows to accelerate convergence, has recently been applied to Langevin sampling by Christie et al. \cite{christie2025speeding} on the torus, who showed that exponentially mixing flows can make both the dissipation time and the mixing time arbitrarily small.

Motivated by these results, we now examine the effect of such flows on the relative entropy. Our observation is that for \textbf{warm-start} initial data, where the ratio $\rho/\pi$ is bounded above and below, the relative entropy is equivalent to the square of $L^2(\pi)$ norm of $\rho/\pi-1$. Hence the decay of the relative entropy can be studied through the $L^2(\pi)$ decay of the semigroup generated by $\mathcal{L}_A = -\nabla U\cdot\nabla + Av\cdot\nabla + \Delta$. Following the resolvent method \cite{Wei18Mix} together with the spectral gap result by Franke \cite{franke2010behavior}, we show that the asymptotic decay rate of the relative entropy as $A\to\infty$ is given by  
$$
r(v;\pi)=\inf\left\{\frac{\int|\nabla\phi|^2\pi}{\int|\phi|^2\pi} : \phi\in H^1(\pi),\ \int \phi \pi =0,\ v\cdot\nabla\phi=i\mu\phi\ \text{for some }\mu\in\mathbb{R}\right\}.
$$
If $r(v;\pi)=\infty$ (i.e., $v$ is relaxation enhancing), the relative entropy decays arbitrarily fast as $A$ increases. However, this conclusion depends on the existence of a flow that is relaxation enhancing with respect to $\pi$, and building such flows directly is not easy. To overcome this, we introduce \textbf{asymptotic relaxation enhancing} flows: sequences $v_n$ with $r(v_n;\pi)\to\infty$. Such flows can be obtained by scaling cellular flows on the torus and pushing them forward by a diffeomorphism that sends Lebesgue measure to $\pi$. By analyzing the quantity $r(v_n;\pi)$, we show that scaling cellular flows on the torus and pushing them forward by a diffeomorphism that sends Lebesgue measure to $\pi$ yields such sequences. This construction gives a simple qualitative explanation for the phenomenon studied in Iyer, Xu and Zlato\v{s} \cite{iyer2021convection}, where a more involved stochastic method provided quantitative bounds on the dissipation time and effective diffusivity. This yields explicit families of flows for which the relative entropy can be made to decay arbitrarily fast. On the other hand, for non-compact domains such as $\mathbb{R}^d$, the situation is more sophisticated. Zlato\v{s} \cite{zlatovs2010diffusion} characterized dissipation-enhancing flows in two dimensions, but his construction for explicit example flows relies on periodization and does not directly extend to general weighted measures. Instead, we use a Lyapunov function method to control the behavior at infinity. This approach works directly with the weighted measure $\pi$ and does not require periodization. Under natural growth conditions on $U$ (e.g., Gaussian tails), the same conclusions hold.

Based on these developments, we now summarize our main findings:
\begin{enumerate}
     \item \textbf{Sharp rate via $r(v;\pi)$.} When $r(v;\pi)$ is finite, we prove that $\Psi_A$ converges to $r(v;\pi)$ as $A\to\infty$. This sharp rate follows from combining the resolvent estimate of Wei \cite{Wei18Mix} with the spectral gap asymptotics of Franke et al. \cite{franke2010behavior}.

    \item \textbf{Asymptotic relaxation-enhancing flows in weighted space and extension to full space} For any smooth target density $\pi$ on $\mathbb{T}^2$, we construct a sequence of flows $v_n$ with $r(v_n;\pi)\to\infty$ using cellular flows and a diffeomorphism. This construction provides a simple qualitative proof of the relaxation enhancement phenomenon studied in \cite{iyer2021convection}. Hence the relative entropy can be made to decay arbitrarily fast by an explicit family of such flows.
     Under suitable growth conditions on $U$ (e.g., Gaussian tails), the same conclusions hold on non-compact domains. The proof uses a Lyapunov function method, which differs from the periodization approach used by Zlato\v{s} \cite{zlatovs2010diffusion} for constructing relaxation enhancing flows.

    \item \textbf{Relative entropy decay.} For warm-start initial data, a relaxation-enhancing flow (i.e., $r(v;\pi)=\infty$) or a sequence of asyptotic relaxation enhancing flows make the relative entropy decay arbitrarily fast when $A$ is large enough. 
    This extends the $L^2$ results of \cite{CKRZ08, zlatovs2010diffusion} and determines the limiting decay rate of the relative entropy

\end{enumerate}

The paper is organized as follows. Section \ref{section:setup} sets up the framework and recalls functional inequalities. Section \ref{section:speedup} presents the resolvent estimate and the relation between $\Psi_A$ and $r(v;\pi)$. Section \ref{section:examples} constructs asymptotic relaxation enhancing flows on the torus and the full space.

\section{Setup and preliminaries}\label{section:setup}

The Fokker-Planck equation for \eqref{eq:langevinFlow} is given by
\[
\partial_t\rho=\nabla\cdot((\nabla U(x)-Av)\rho)+\Delta \rho.
\]
Consider the invariant measure
\[
\pi\propto \exp(-U(x)).
\]
If
\begin{gather}\label{eq:weighteddivfree}
\nabla\cdot(v(x)e^{-U(x)})=0,
\end{gather}
the invariant measure is unchanged.

To obtain flows satisfying \eqref{eq:weighteddivfree}, one may make use of the differential forms that generates divergence free field (e.g., the streamfunction in 2D).
If the domain is globally simply connected, a divergence free field $\tilde{v}$ satisfies
\[
\iota_{\tilde{v}}(\omega)=d\beta
\]
for some $(d-2)$-form $\beta$, where $\omega$ is the volume form. In particular, in $\R^d$, one has
\[
\sum_{i=1}^n (-1)^{i-1} \tilde{v}^i \,
dx^1 \wedge \cdots \wedge \widehat{dx^i} \wedge \cdots \wedge dx^n=d\beta.
\]
Hence, using the $d-2$ forms $\beta$, one may construct the desired weighted divergence free $v$
using $\pi^{-1}d\beta$. Since $\pi$ often decays fast, one may add a moderating function $\chi$ to obtain a velocity field that does not grow too fast,  in particular,
\[
\iota_{\pi v}(\omega)=d(\chi \beta).
\]
For example, one may choose $\chi=\pi$ to obtain a field that is comparable to the field corresponding to $d\beta$. 
Of course, if the domain is not globally simply connected like the torus, such $\beta$ may not exist. Another possible way to obtain divergence free field for the weighted measure is to use the pushforward map of diffeomorphisms (see section \ref{subsec:retorus}).

\subsection{Entropy dissipation and logarithmic Sobolev inequality}

We consider the relative entropy (KL divergence) to measure the difference between two probability measures, which is defined as
\begin{gather}\label{eq:relativeentropy}
\cH(\mu\mid \nu) = \left\{
\begin{aligned}
  &\int_{\Omega} \log \frac{\mathrm{d} \mu}{\mathrm{d}\nu} \mathrm{d}\mu, & \text{if} ~ \mu \ll \nu,\\
 & \infty, & \text{else}.
\end{aligned}\right.
\end{gather}
Here, $\frac{\mathrm{d} \mu}{\mathrm{d}\nu}$ denotes the Radon-Nikodym derivative of $\mu$ with respect to $\nu$. 
The relative entropy is a non-negative quantity by Jensen's inequality and achieves zero only if $\mu= \nu$. Moreover, the relative entropy is jointly convex in its arguments and is lower semicontinuous with respect to weak convergence. In our problem, since both $\rho$ and $\pi$ have smooth densities so that the relative entropy is reduced to
\begin{gather}
H(t):=\mathcal{H}(\rho(t)|\pi)=\int \left(\log \frac{\rho}{\pi}\right)\rho\,dx.
\end{gather}
We would like to consider the quantity
\begin{gather}
q=\rho/\pi
\end{gather}
and then
\begin{equation}\label{eq:Ht}
H(t)=\mathcal{H}(\rho(t)|\pi)=\int q\log q \pi\,dx.
\end{equation}
It is not hard to derive the equation for $q$ as
\begin{equation}\label{eq:LA}
\partial_tq=(\nabla U-Av+2\nabla\log \pi)\cdot \nabla q+\Delta q
=(-\nabla U-Av)\cdot \nabla q+\Delta q=:\cL_A q.
\end{equation}
In particular, we denote
\begin{gather}\label{eq:L0}
\cL_0:=-\nabla U\cdot \nabla +\Delta \,.
\end{gather}

Later, we will always assume the following regularity on the potential $U$.
\begin{assumption}\label{ass:Ureg}
The potential $U\in C^2(\Omega)$. If $\Omega=\R^d$, we further assume the confining condition 
\begin{equation}\label{eq:confiningRd}
\lim_{|x|\to\infty} \left(\frac{|\nabla U|^2}{2}-\Delta U(x)\right)=+\infty.
\end{equation}
\end{assumption}

We note several simple properties about $\cL_0$.
\begin{itemize}
\item $\cL_0$ is self-adjoint in $L^2(\pi)$ and $-\cL_0$ is m-accretive, which is due to the fact that for any $f,~g\in L^2(\pi)$,
\begin{gather}\label{eq:symmetryL0}
\int f\cL_0 g\pi\,dx=-\int\nabla f\cdot\nabla g\pi\,dx.
\end{gather}
\item Suppose Assumption \ref{ass:Ureg} holds. 
Then the spectrum of $\cL_0$ is discrete (see~\cite[Chapter 4]{pavliotis2014} for reference).
\end{itemize}
It is well known that if Assumption \ref{ass:Ureg} holds, then one also has
\[
H^1(\pi) \hookrightarrow L^2(\pi) \quad \text{compactly}.
\]

We then take the time derivative and get
\[
\frac{d}{dt}H(t)=\int \log q(\cL_0 q-Av\cdot\nabla q)\pi\,dx
=-\int \nabla\log q\cdot\nabla q\pi\,dx+A\int q v\pi\cdot \nabla\log q\,dx.
\]
For the second term, one has
\[
\int q v\pi\cdot \nabla\log q\,dx
=\int v\pi \cdot \nabla q\,dx=-\int \nabla\cdot(v\pi) q\,dx=0.
\]
Hence, one has
\begin{equation}\label{eq:Htest1}
\frac{d}{dt}H(t)=-\int |\nabla\log q|^2 q\pi\,dx.
\end{equation}
This formula is the same as the original entropy dissipation without $v$.

Here, we recall the log-Sobolev inequality. Consider the entropy of the square of a Lipschitz function $f$ with respect to $\pi$, defined as:
\begin{equation}
	\operatorname{Ent}_{\pi}(f^2) = \int_{\R^d} f^2 \log(f^2) d \pi - \left( \int_{\R^d} f^2 d\pi \right) \log \left( \int_{\R^d} f^2 d\pi \right).
\end{equation}
We say $\pi$ satisfies the log-Sobolev inequality (LSI) with constant $\lambda$ if for all such $f$, one has 
\begin{equation}\label{eq:LSI}
	\operatorname{Ent}_{\pi}(f^2) \leq \frac{2}{\lambda} \int_{\R^d} |\nabla f|^2 d\pi\,.
\end{equation}
We take $f=\sqrt q$ and get 
\begin{align}
H(t)&=\operatorname{Ent}_\pi(f^2)\,,\\
\int|\nabla\log q|^2q\pi\,dx&=4\int |\nabla f|^2\pi\,dx\,.
\end{align}
If the log-Sobolev inequality holds for the measure $\pi$, we further obtain
\[
\frac{d}{dt}H(t)\le -2\lambda H(t),
\]
and thus the relative entropy decays exponentially with rate $2\lambda$. However, $\lambda$ itself can be extremely small, a standard example is a double well potential, where the process spends long periods in one well and transitions between wells are rare (that is, the metastability phenomenon).

\subsection{Warm-start initial data}\label{sec:warmup}
In order to enhance the convergence rate, one would ideally seek an improved log-Sobolev inequality with a larger constant. A natural intuition is that the presence of high-frequency components in the initial data leads to stronger dissipation, as such modes are rapidly attenuated by the diffusion operator. This indicates the possibility of a larger effective log-Sobolev constant.

 
When $\lambda$ is small, convergence can be slow, so it is natural to ask whether the LSI constant can be improved by restricting to functions with a large $H^1$ norm. High frequencies are damped more quickly by diffusion, and a larger $H^1$ norm indicates more oscillation. The following example shows that this is not enough: the sharp constant stays the same even among functions with arbitrarily large $H^1$ norm.

Consider the following function:
\[
f_{\beta}(\mathbf{x}) = e^{\frac{\beta x_1}{2} - \frac{\beta^2}{4}}.
\]
It is easy to check that $\int_{\mathbb{R}^2}f_{\beta} ^2\,d\mu=1$, where $\mu$ is the standard Gaussian measure on $\mathbb{R}^2$ with $d\mu=\frac{1}{2\pi}e^{-\frac{|x|^2}{2}}\,dx$.
By simple calculation, one has $\|f_{\beta}\|^2_{H^1}=1+\frac{\beta^2}{4}$, which tends to infinity as $\beta\rightarrow \infty$. On the other hand, $f_\beta$ is an extremizer for the Gaussian logarithmic Sobolev inequality, and in fact
\[
\int f_{\beta}^2\log f_{\beta}^2 d\mu
=2\int |\nabla f_{\beta}|^2d\mu.
\]
Thus, even among functions with arbitrarily large $H^1$ norm, the sharp logarithmic Sobolev constant is still attained. Therefore, restricting the logarithmic Sobolev inequality to a class of functions with large $H^1$ norm cannot improve the optimal constant.
We notice that in our problem, we should have  $q\to 1$ as $t\to \infty$. Hence we may write
\begin{align}\label{e:qexpansion}
q=1+h\,,
\end{align}
where $h$ is small for large time and $h$ satisfies $\int h\pi\,dx=0$ due to mass conservation. In this regime, applying the Taylor expansion, we obtain
\begin{align}\nonumber
q\log q = h+\frac{h^2}{2}+o(h^2)\,.
\end{align}
Integrating against $\pi$, we get
\begin{align*}
\int q\log q \pi\,dx=\frac{1}{2}\int h^2 \pi\,dx+o(\|h\|_{L^2(\pi)}^2)
\end{align*}

Hence, in the near equilibrium regime, the relative entropy is comparable to the squared $L^2(\pi)$ norm. This means that classical relaxation enhancing methods, originally developed in the $L^2$ framework on compact manifolds, can be applied to study the entropy decay under a large drift. We therefore restrict in this paper to warm-start initial data as follows.

\begin{assumption}\label{ass:warmstart}
The initial distribution $\rho_0$ satisfies
\begin{gather}
 0<c_1\le  q(0)=\frac{\rho_0}{\pi}\le c_2<\infty. 
\end{gather}
\end{assumption}
Note that it must hold $0<c_1\leq 1\leq c_2$, since $\int q(0)\pi dx=1$.
This bound can be preserved by the maximum principle. We state it as the following lemma and omit the proof.
\begin{lemma}\label{l:warmstart}
If the initial distribution $\rho_0$ satisfies Assumption \ref{ass:warmstart}, then for all $t\ge 0$, one has
\[
c_1\le q(t) \le c_2.
\]
\end{lemma}

If we expand $q$ as in~\eqref{e:qexpansion}, the LSI is equivalent to
\begin{equation}\label{eq:Htest2}
2\lambda \int (1+h)\log(1+h)\pi dx\le \int \frac{|\nabla h|^2}{1+h}\pi \,dx.
\end{equation}
On the other hand, using Lemma~\ref{l:warmstart} and Taylor expansion, we have
\begin{align}
h+\frac{1}{2c_2}h^2\leq (1+h)\log(1+h)=h+\frac{f''(\xi)}{2}h^2\leq h+\frac{1}{2c_1}h^2\,,
\end{align}
where $f(x)=(1+x)\log(1+x)$ and $\xi$ lies in between $0$ and $h$. Also we have the following
\begin{align*}
\frac{|\nabla h|^2}{c_2}\leq \frac{|\nabla h|^2}{1+h}\leq\frac{|\nabla h|^2}{c_1}\,.
\end{align*}
Hence, to prove~\eqref{eq:Htest2}, it suffices to consider the Poincar\'e inequality
\begin{equation}\label{eq:Htest3}
\kappa\int h^2 \pi dx\le \int |\nabla h|^2 \pi dx\,.
\end{equation}
With this, we can choose 
\[
\lambda=\frac{\kappa c_1}{c_2}
\]
in~\eqref{eq:Htest2}\,. In this case, the improvement in Poincar\'e inequality then passes to the log-Sobolev inequality naturally.

\begin{remark}
In this work, we restrict ourselves to warm-start initial data and focus on the Poincar\'e inequality, which effectively captures the behavior of the LSI in the near-equilibrium regime. When the initial data is not warm-start, the density ratio $q$ typically decays to zero as $|x|\to \infty$, although it may remain close to $1$ in a central region. Understanding how the flow interacts with such spatial inhomogeneity, and whether it can enhance the LSI constant in this setting, remains an interesting open question.
\end{remark}

\subsection{The Wasserstein distances and transportation inequalities}

 Let $\cP(\Omega)$ be the set of all probability measures on $\Omega$, and $\cP_p(\Omega)$, indexed by some $p\ge 1$, be the set of probability measures having finite $p$-moment, i.e.,
\begin{equation}\label{eq:P2}
\mathcal{P}_{p}\left(\Omega\right):=\left\{\mu \in \mathcal{P}\left(\Omega\right): \int_{\Omega}|x|^p \mu(d x)<\infty\right\}.
\end{equation}
Clearly, $\cP_p\subset \cP_2$ for $p\ge 2$.
For $p\ge 1$ and any $\mu, \nu\in\mathcal{P}_{p}(\mathbb{R}^{d})$, the $p$-Wasserstein distance is defined by
 \begin{equation*}
 \mathcal{W}_{p}(\mu, \nu):=\left(\inf _{\gamma \in \Pi(\mu, \nu)} \int_{\Omega \times \Omega}d(x, y)^{p} \gamma(d x, d y)\right)^{1/p},
 \end{equation*}
where $d(x, y)$ is the distance between $x$ and $y$, and $\Pi(\mu, \nu)$ is the set of couplings of $\mu$ and $\nu$ (i.e., joint distributions with marginals to be $\mu$ and $\nu$ respectively). The space $\mathcal{P}_{p}\left(\Omega\right)$ is a Polish space under the $p$-Wasserstein metric.

A class of important functional inequalities is the transportation inequalities \cite{talagrand1996transportation,bobkov2001hypercontractivity}. The Talagrand transportation inequality gives the control for $W_2(\mu,\nu)$ if $\nu$ satisfies a log-Sobolev inequality with constant $\lambda$ as in \eqref{eq:LSI} \cite{talagrand1991new,otto2000generalization}:
\begin{equation}
W_1(\mu, \nu)\le W_2(\mu,\nu) \leq \sqrt{\frac{2}{\lambda} \cH(\mu |\nu)}.
\end{equation}

Moreover, using the weighted Csisz\'ar-Kullback-Pinsker inequality \cite{bolley2005weighted}, one has
\begin{equation}
W_1(\mu,\nu) \leq C_{\nu}\sqrt{ \cH(\mu |\nu)},
\end{equation}
for any $\nu$ satisfying the following tail behavior:
\begin{equation}\label{eq:W1condition}
C_\nu:=2 \inf _{\alpha>0}\left(\frac{1}{2 \alpha}\left(1+\log \int_{\Omega} e^{\alpha|x|^2} d \nu(x)\right)\right)^{\frac{1}{2}}<+\infty.
\end{equation}

As an extra comment, we recall the Pinsker's inequality \cite{pinsker1964information,bolley2005weighted}, which claims that the total variation norm (and also some weighted versions) can be controlled by the square root of the relative entropy
\begin{gather}
TV(\mu, \nu)\le \sqrt{2\cH(\mu\mid \nu)},
\end{gather}
where
\[
TV(\mu, \nu)=\int_{\Omega} d|\mu-\nu|=2 \sup_{B\subset \Omega}|\mu(B)-\nu(B)|
\,,\]
is the standard total variation norm.
This inequality holds without any restrictions on the measures involved.

\section{Accelerating sampling by flows}\label{section:speedup}

We need to consider the effect of the transport operator
$Av\cdot\nabla$. In particular, the eigenfunctions of this operator become essential if $A\to\infty$. In this section, we explore how such an operator could affect the mixing.


\subsection{Relaxation enhancing flows}

Consider the space
\begin{gather}\label{eq:mzspace}
\cX=\left\{\phi\in L^2(\pi): \int \phi \pi dx=0 \right\}.
\end{gather}
Clearly, $\cX$ is an invariant space of the semigroup generated by $\cL_A$ for any $A\ge 0$.
Depending on the initial data we consider, one may focus on more restricted subspaces of $L^2(\pi)$, equipped with the same $L^2(\pi)$ norm. 
Hence, for the general purpose, we will assume that there is a closed subspace $\cY$ of $X$
\begin{gather}\label{eq:mzsspace}
\cY\subset \cX,
\end{gather}
such that it is invariant under the semigroup of $\cL_A$ for any $A\ge 0$.
If we consider the general initial data $\rho_0$, we have $\cY=\cX$.
Depending on the applications, $\cY$ may be smaller.  For example, for the flat measure on torus $\T\times \R$ and $v=(y, 0)$ is the shear flow, one may take $\cY$ to be the set of functions that has mean zero in the $x$ direction for any $y$.

Let
\begin{gather}\label{eigenspace}
\dot{\mathcal{I}}(v; \cY, \pi) :=\left\{\phi\in \cY:  \phi\in H^1(\pi), v\cdot\nabla \phi=i \mu \phi, \mbox{ for some } \mu\in \R \right\}
\end{gather}
be the eigenfunctions of $v\cdot\nabla$ in $\cY$ that are $H^1$. Any eigenfunction with nonzero eigenvalue is mean zero.
If $v\cdot \nabla$ has an eigenfunction $\phi$ with eigenvalue $\lambda$, since $v\cdot \nabla$ is skew-symmetric, $\lambda=i \mu$
for some $\mu\in \R$. Hence, $|\phi|$ is also an eigenfunction with zero eigenvalue. If $\phi$
is in $H^1$, so is $|\phi|$. For real functions, the $H^1$ norm of $|\phi|$ is the same as $\phi$. However, for complex functions, this is not generally true. For the convenience of the notation, unless we emphasize the dependence on $\cY$ or $\pi$, we will simply use $\dot{\mathcal{I}}(v)$ to indicate this set of $H^1$ eigenfunctions.

The concept of relaxation enhancing flows was introduced in \cite{CKRZ08} for the ability to enhance the mixing, and can be easily adopted for our purpose.
\begin{definition}[Relaxation enhancing]
A flow $v$ satisfying~\eqref{eq:weighteddivfree} is \emph{relaxation enhancing} if for every $\tau > 0$ and $\delta > 0$, there exists $A(\tau,\delta) > 0$ such that for all $A \ge A(\tau,\delta)$,
\[
\|e^{\tau \mathcal{L}_A}\|_{\cY \to \cY} < \delta.
\]
\end{definition}
The characterization of the relaxation enhancing flow in the paper by Constantine et al. \cite{CKRZ08} is unchanged for our weighted space.
\begin{theorem}
The flow $v$ is relaxation enhancing if and only if $v\cdot\nabla$ has no eigenvalues in $\cY$ or equivalently $\dot{\mathcal{I}}(v)=\emptyset$.
\end{theorem}

For the warm-start data, one could conclude the following.
\begin{proposition}
Suppose $v$ is relaxation enhancing in $\cY\subset L^2(\pi)$. Suppose that $q_0$ satisfies Assumption \ref{ass:warmstart} with $q_0-1\in \cY$. Then, for any given $B>0$ and $t>0$, there is $A_0$, whenever $A>A_0$, one has
\begin{equation}
H(t)\le H_0\exp(-B t),
\end{equation}
where $H(t)$ is defined as \eqref{eq:Ht}.
\end{proposition}

The proof is in spirit the same as the one in \cite{CKRZ08}. Here, we just sketch the argument and explain some main ideas.
Under the warm start condition, the entropy dissipation formula \eqref{eq:Htest1} and the equivalence between the LSI and the Poincar\'e inequality (see Section~\ref{section:setup}, \eqref{eq:Htest2}, \eqref{eq:Htest3}) yield the bound
\begin{gather}
\frac{d}{dt}H(t)
=-\int |\nabla \log q|^2 q\pi \,dx
\le -C\frac{\int |\nabla h|^2\pi\,dx}{\int |h|^2\pi\,dx}H(t).
\end{gather}
Hence,
\begin{equation}\label{eq:Htrbound}
H(t) \le H(0) \exp\!\Bigl( -C \int_0^t r(s)\,ds \Bigr),
\qquad
r(s) = \frac{\displaystyle\int |\nabla h(s)|^2\,\pi(dx)}{\displaystyle\int h(s)^2\,\pi(dx)},
\end{equation}
where $h = q-1$ and $C$ depends only on the constants $c_1,c_2$ from Assumption~\ref{ass:warmstart}.
The acceleration occurs because for a relaxation enhancing flow $v$, the transport operator $v\cdot\nabla$ has no $H^1$ eigenfunctions in the invariant subspace $\cY$.
By a RAGE type theorem (see \cite{CKRZ08}) this drives the solution toward high frequency modes, making $r(s)$ large.
In fact, for any given $t>0$ and any $B>0$, the amplitude $A$ can be chosen sufficiently large so that $\int_0^t r(s)\,ds \ge B t$. We refer the interested readers to \cite{CKRZ08}.
Consequently, the relative entropy can be made to decay arbitrarily fast.

Note that $A_0$ here depends on the two constants in Assumption \ref{ass:warmstart} and also $t$. For different $t$, the constant $A_0$ needed should be different.

\subsection{An example of relaxation enhancing flows on torus}\label{subsec:retorus}

To make the abstract definition concrete, in this subsection we construct an explicit relaxation-enhancing flow on $\mathbb{T}^2$ for a smooth target density $\pi$ when $\cY=\cX$, so that the improvement of the rate works for all possible warm-start initial data.

We first note the following simple fact, which follows by direct verification using chain rule.
\begin{lemma}
Let $Z: \Omega\to \Omega$ be a diffeomorphism with $z=Z(x)$. Then, the flow $v$ is transformed to
\[
\bar{v}=:Z_{\#}v=(v\cdot \nabla_x Z)\circ Z^{-1}.
\]
Moreover, the operator $\bar{v}\cdot\nabla_z$ has the same eigenvalues as $v\cdot\nabla_x$, and their eigenfunctions are related by
\begin{gather}
\bar{\phi}(Z(x))=\phi(x).
\end{gather}
\end{lemma}

Let $\pi$ be a measure with a smooth density (still denoted by $\pi$) on $\T^2$.
In \cite[section 6]{CKRZ08}, a smooth flow $v$ on $\T^2$ has been constructed by Proposition $6.2$. This flow satisfies the following properties:
\begin{itemize}
\item This flow preserves the measure $dx dy$, or $\mathrm{div}( v)=0$.
\item This flow is smooth.
\item This flow has discontinuous eigenfunctions, hence none of its eigenfunctions belong to $H^1$.
\end{itemize}
Our aim is to construct a flow with similar properties, but with respect to a non-flat measure on $\mathbb{T}^2$. To do this, the idea is to construct a smooth diffeomorphism $Z: \T^2\to \T^2$ such that the distribution $F$ is transformed to $\pi \,dx dy$.
Then, the flow $\bar{v}=Z_{\#}v$ will satisfy
\[
\mathrm{div}(\pi \bar{v})=0.
\]
Since $Z$ is a diffeomorphism, the eigenfunctions of $\bar{v}\cdot \nabla$
will also be discontinuous and thus not in $H^1(\pi)$. Hence, $\bar{v}$ will be a desired relaxation enhancing flow
on $\T^2$.

The diffeomorphism can be constructed in many ways. The optimal transport map is one such example.
Another explicit construction is as follows the idea of this construction is very similar to the construction of Section 6 of \cite{CKRZ08}.
\begin{proposition}\label{prop:from1to2}
Assume $F$ and $\pi$ are both smooth function in $\mathbb{T}^2$. For a flow $v$ satisfies $\mathrm{div}(F v)=0$, there exists a smooth diffeomorphism $Z: \T^2\to \T^2$, so that $\mathrm{div}(\pi \bar{v})=0$ while $\bar{v}=Z_{\#}v$. Moreover, all the eigenvalues of $\bar{v}\cdot\nabla$ and $v\cdot\nabla$ are the same, if $\phi$ is an eigenfunction of $v\cdot\nabla$, then $Z_{\#}\phi$ is an eigenfunction of $\bar{v}\cdot\nabla$, and vice versa.
\end{proposition}
\begin{proof}
One can easily get the statement for the eigenvalues and eigenvectors by the definition of diffeomorphism. We only need to construct such a diffeomorphism. 

Consider the marginals for $F$ and $\pi$
\begin{gather}
\bar{F}(x)=\int_0^1F(x, y)\,dy,\quad \bar{\pi}(p)=\int_0^1 \pi(p, q)\,dq,
\end{gather}
and their conditional cumulative distribution functions
\begin{gather}
G(y| x)=\frac{1}{\bar{F}(x)}\int_0^y F(x, y')dy',\quad H(q| p)=\frac{1}{\bar{\pi}(p)}\int_0^q \pi(p, q')dq'.
\end{gather}
Define $P(x)$ through
\[
\bar{H}(P(x))=\int_0^x \bar{F}(s)\,ds,\quad \bar{H}(x)=\int_0^x \bar{\pi}(s)\,ds,
\]
then we have
\[
P'(x) \bar{\pi}(P(x))=\bar{F}(x).
\]
Thus, with the following $Q$,
\[
Q(x, y)=H(\cdot | P(x))^{-1}(G(y| x)).
\]
The mapping
\[
Z(x, y)=(P(x), Q(x, y))
\]
satisfies the desired property. It is easy to verify that
\[
\det(D Z)=\frac{F(x,y)}{\pi(P(x), Q(x, y))}.
\]
\end{proof}

\subsection{Flows with nonconstant $H^1$ eigenfunctions}

Although the previous example shows that relaxation enhancing flows exist on $\mathbb{T}^2$, their construction is often very tricky and restrictive, especially on $\mathbb{R}^d$ with a decay weight. In particular, on the full space a steady relaxation enhancing flow with respect to a Gaussian measure would typically have infinite energy. Indeed, if $v = \pi^{-1}\nabla^\perp\psi$, then $\psi$ is a first integral, and the condition $\psi\notin\dot{\mathcal{I}}(v,\pi)$ forces $\int|v|^2\pi^3 = \infty$, which is far too restrictive. Therefore, it is more practical to study flows that possess nontrivial $H^1$ eigenfunctions and to quantify the best possible dissipation enhancement.

In this section, we explore how well a flow could enhance the dissipation if it has $H^1$ eigenfunctions.
For a given flow $v$, we define the following number
\begin{gather}\label{eq:defr}
r(v; \cY, \pi):=\inf_{\phi\in \dot{\mathcal{I}}(v; \cY, \pi)}\frac{\int |\nabla\phi|^2\pi dx}{\int |\phi|^2\pi dx}
\ge \kappa,
\end{gather}
where $\dot{\mathcal{I}}(v;\cY, \pi)$ is defined in \eqref{eigenspace} and $\kappa$ is the Poincar\'e constant in $L^2(\pi)$ with mean zero.
As above, we will use the notation $r(v)$ to indicate this number unless we emphasize the dependence on $\cY$
or $\pi$.
\begin{lemma}
If $\dot{\mathcal{I}}(v)$ contains a nonzero element, then $r(v)$ can be achieved.
\end{lemma}
\begin{proof}
By the definition of the infimum, we can take a sequence $\phi_n\in \dot{I}(v)$ such that $\int |\phi_n|^2\pi\,dx=1$
and
\[
\int|\nabla\phi_n|^2\pi \le r(v)+1/n.
\]
Then, one can draw a subsequence $\{\phi_{n_k}\}$ such that it converges weakly in $H^1(\pi)$ and strongly in $L^2(\pi)$ to some function $\phi$. Consequently, it is easy to see that $\phi\in \dot{\mathcal{I}}(v)$ and  
\[
\int|\nabla\phi|^2\pi \le r(v).
\]
Hence, $\phi$ is a minimizer.
\end{proof}

As pointed out in \cite{berestycki2005elliptic, franke2010behavior}, the number $r(v)$ is closely related to the spectral gap of $\cL_A$ as $A\to\infty$. 
It is easy to see that $\cL_A^{-1}$ is a compact operator on $\cY$ for any $A\ge 0$. Hence, $\cL_A$ has discrete eigenvalues.
Note that the principal eigenvalue of $\cL_A$ on $L^2(\pi)$ is zero with the eigenfunction $\phi=1$. As we consider the space $\cY$, the principal eigenvalue is actually the second eigenvalue of $\cL_A$ in $L^2(\pi)$, so the principal eigenvalue could be complex. Consider the set $\Lambda_A$ of eigenvalues of $-\cL_A$ and 
\begin{gather}
m_A:=\inf \{\mathrm{Re}(\lambda): \lambda\in \Lambda_A \}.
\end{gather}
If there is one eigenvalue whose real part equals $m_A$, then one calls this the principal eigenvalue. If the principal eigenvalue exists, we denote $(\lambda_A, \phi_A)$ the corresponding eigen-pair.
\begin{gather}
-\cL_A\phi_A=\lambda_A \phi_A,\quad \mathrm{Re}(\lambda_A)=m_A.
\end{gather}

\begin{lemma}\label{lmm:principal}
If $v\in L^2(\pi)$, then the principal eigen-pair exists, and the spectral gap satisfies
\[
m_A = \operatorname{Re}(\lambda_A) = \frac{\displaystyle\int |\nabla \phi_A|^2 \,\pi(dx)}{\displaystyle\int |\phi_A|^2 \,\pi(dx)} \ge \kappa .
\]
\end{lemma}

\begin{proof}
Recall that $m_A = \inf\{\operatorname{Re}(\lambda) : \lambda \in \Lambda_A\}$.  Choose a sequence of eigenpairs $(\lambda^{(k)},\phi^{(k)})$ with
$\|\phi^{(k)}\|_{L^2(\pi)} = 1$ and $\operatorname{Re}(\lambda^{(k)}) \to m_A$.  For each $k$, testing the eigenvalue equation
\[
-\cL_0 \phi^{(k)} + A v \cdot \nabla \phi^{(k)} = \lambda^{(k)} \phi^{(k)}
\]
against $\overline{\phi^{(k)}}$ in $L^2(\pi)$ gives, after integration by parts,
\[
\int |\nabla \phi^{(k)}|^2 \,\pi(dx) + A \int (v \cdot \nabla \phi^{(k)}) \overline{\phi^{(k)}} \,\pi(dx)
= \lambda^{(k)} .
\]
Taking real parts shows $\operatorname{Re}(\lambda^{(k)}) = \int |\nabla \phi^{(k)}|^2 \,\pi(dx)$.
Thus $\{\phi^{(k)}\}$ is bounded in $H^1(\pi)$.  By the compact embedding $H^1(\pi) \hookrightarrow L^2(\pi)$,
we can take a subsequence (still denoted $\phi^{(k)}$) that converges weakly in $H^1(\pi)$
and strongly in $L^2(\pi)$ to some $\phi \in H^1(\pi)$ with $\|\phi\|_{L^2(\pi)} = 1$.
Weak lower semicontinuity of the norm gives
\[
\int |\nabla \phi|^2 \,\pi(dx) \le \liminf_{k\to\infty} \int |\nabla \phi^{(k)}|^2 \,\pi(dx) = m_A .
\]

Now take any $\psi\in C_c^\infty(\Omega)$ with $\int \phi \overline{\psi}\,\pi(dx)\neq0$
(such a $\psi$ exists because $\phi\not\equiv0$).  For all large $k$ we also have
$\int \phi^{(k)}\overline{\psi}\,\pi(dx)\neq0$.  Writing the weak formulation
\[
\int \nabla \phi^{(k)}\!\cdot\!\nabla\overline{\psi}\,\pi(dx)
- A\int \phi^{(k)}(v\cdot\nabla\overline{\psi})\,\pi(dx)
= \lambda^{(k)}\int \phi^{(k)}\overline{\psi}\,\pi(dx),
\]
the left-hand side converges.  Because the integrals $\int \phi^{(k)}\overline{\psi}\,\pi(dx)$
converge to a nonzero limit, the sequence $\lambda^{(k)}$ must converge to some $\lambda$.
Passing to the limit in the weak formulation shows that $-\cL_0\phi + A v\cdot\nabla\phi = \lambda\phi$
in the weak sense.  Testing with $\psi=\phi$ (after a standard approximation) gives
$\operatorname{Re}(\lambda) = \int|\nabla\phi|^2\pi(dx) \le m_A$.
Since $m_A$ is the infimum of real parts, we obtain $\operatorname{Re}(\lambda)=m_A$,
and $\phi$ is a principal eigenpair.
\end{proof}

In \cite{franke2010behavior}, it has been proved that the spectral gap tends to $r(v)$ as $A\to\infty$ for the problem on a compact manifold with flat measure. In fact, the argument can apply to the operator $\cL_A$ on $\cY$. In particular, one has the following claim corresponding to \cite[Theorem 2]{franke2010behavior}.
\begin{proposition}\label{pro:continuityeigenval}
Suppose $\dot{\mathcal{I}}(v)$ contains a nonzero element and the eigenfunction corresponding to the minimizer for $r(v)$ is $i\mu$. Then, for any $\epsilon>0$, there is $A_{\epsilon}>0$ such that for all $A>A_{\epsilon}$, $\cL_A=\cL_0-Av\cdot\nabla$
has an eigenvalue in the ball $B(-r(v)-iA\mu, \epsilon)$.
\end{proposition}

Clearly, with this observation, for any $\delta>0$, if the initial data is taken to the be corresponding eigenvalue, one has
\begin{gather}
\|\exp(t\cL_A)\phi_0\|_2 \ge \|\phi_0\|_2 \exp(-(r(v)-\delta)t).
\end{gather}
With the relation between the entropy and the $L^2$ norms for warm-start data, if the corresponding eigenfunction is bounded, one can construct an initial data such that the rate of the decay for the entropy cannot be improved
larger than $r(v)$.

However, following the ideas in Wei \cite{Wei18Mix}, we can show that $r(v)$ is in fact the sharp rate for decay.
For $\cL_A$ in \eqref{eq:LA}, one could introduce
\begin{gather}\label{eq:psiA}
\Psi_A:=\inf\{\|(\cL_A-i\lambda)f\|_{L^2(\pi)} : f\in D(\cL_A),  \lambda\in \R, \|f\|_{L^2(\pi)}=1 \}.
\end{gather}
Clearly, $\Psi_A\ge \kappa$. It has been shown in \cite{Wei18Mix} that the solution to the equation
\[
\partial_t\phi=\cL_A\phi
\]
has the following decay bound
\begin{gather}\label{eq:decaybound}
\|\phi(t)\|_{L^2(\pi)}\le \|\phi(0)\|_{L^2(\pi)}\exp(-\Psi_A t+\uppi/2),
\end{gather}
where $\uppi=3.14\cdots$ denotes the usual circle constant, to avoid confusion with the invariant measure $\pi$. This holds {\it for all $t\ge 0$}.
By analyzing the asymptotic behavior of $\Psi_A$, we have the following.
\begin{theorem}\label{thm:REA}
Consider a given flow $v$. Let
$\sigma:=\liminf_{A\to\infty}\Psi_A$. Then,
\begin{enumerate}[(i)]
\item $v$ is relaxation enhancing (or $\dot{\mathcal{I}}(v)=\emptyset$)
if and only if $\sigma=\infty$, and also if and only if $\lim_{A\to\infty}m_A=\infty$.

\item  If $\sigma<\infty$, then $\dot{\mathcal{I}}(v)$
is nonempty and 
\begin{gather}
\sigma=\lim_{A\to\infty}m_A=r(v).
\end{gather}
\end{enumerate}
Moreover, suppose $\sigma<\infty$ and $q_0$ satisfies Assumption \ref{ass:warmstart} with $q_0-1\in \cY$. Then, for any $\delta>0$, there is some $A_0$ such that whenever $A>A_0$, one has
\[
\cH(\rho(t)| \pi)=\int q \log q \pi dx \le C\int |h|^2\pi \,dx \le C_1\exp(-(\sigma-\delta) t),
\]
where $C_1$ depends only on the two constants in Assumption \ref{ass:warmstart}.
\end{theorem}

\begin{proof}
Recall that $v$ is relaxation enhancing if and only if $\dot{\mathcal{I}}(v)=\emptyset$.
If $\dot{\mathcal{I}}(v)\neq\emptyset$, then Proposition~\ref{pro:continuityeigenval}
provides an eigenvalue of $-\cL_A$ with real part at most $r(v)+\varepsilon$ for all large $A$, hence
$\limsup_{A\to\infty} m_A\le r(v)<\infty$.  Taking $f=\phi_A$ in the expression on the right hand side of \eqref{eq:psiA}, one obtains $\Psi_A\le m_A$. Hence, $\sigma<\infty$.
Thus $\sigma=\infty$ forces $\dot{\mathcal{I}}(v)=\emptyset$, i.e., relaxation enhancing.

Conversely, assume $\sigma<\infty$.  Choose sequences $A_k\to\infty$, $\lambda_k\in\mathbb{R}$,
and $f_k\in D(\cL_{A_k})$ with $\|f_k\|_{L^2(\pi)}=1$ such that
$\|(\cL_{A_k}-i\lambda_k)f_k\|_{L^2(\pi)}\le\sigma+1/k$.
Setting $g_k=(\cL_{A_k}-i\lambda_k)f_k$ and testing against $\overline{f_k}$ gives
$\int|\nabla f_k|^2\pi(dx)=\operatorname{Re}\langle g_k,f_k\rangle\le\sigma+1/k$;
hence $\{f_k\}$ is bounded in $H^1(\pi)$.  By compactness, a subsequence converges weakly
in $H^1(\pi)$ and strongly in $L^2(\pi)$ to $f_\infty\in H^1(\pi)$ with $\|f_\infty\|_{L^2(\pi)}=1$.
For any $\psi\in C_c^\infty(\Omega)$,
\[
\int\nabla f_k\!\cdot\!\nabla\overline\psi\,\pi(dx)
-A_k\int(v\cdot\nabla f_k)\overline\psi\,\pi(dx)
-i\lambda_k\int f_k\overline\psi\,\pi(dx)=\int g_k\overline\psi\,\pi(dx).
\]
Dividing by $A_k$ and taking limits shows that a subsequence of $\lambda_k/A_k$ converges
to some $\mu$ and $v\cdot\nabla f_\infty=-i\mu f_\infty$ in the weak sense, so
$f_\infty\in\dot{\mathcal{I}}(v)$.  Hence $\dot{\mathcal{I}}(v)\neq\emptyset$, proving that
$\sigma<\infty$ implies non-relaxation-enhancing.  The same argument with the principal
eigenfunctions shows that $\lim_{A\to\infty}m_A=\infty$ is also equivalent to relaxation enhancing.

Now suppose $\sigma<\infty$.  From the sequence above,
$\sigma\ge\int|\nabla f_\infty|^2\pi(dx)\ge r(v)$.
Taking $f=\phi_A$ in the definition of $\Psi_A$ yields $\Psi_A\le m_A$, so
$\sigma\le\liminf_{A\to\infty} m_A$.
Proposition~\ref{pro:continuityeigenval} gives $\limsup_{A\to\infty} m_A\le r(v)$,
and therefore $\sigma=\lim_{A\to\infty} m_A=r(v)$.

For the entropy decay when $\sigma<\infty$, fix $\delta>0$ and choose $A$ large enough
so that $\Psi_A\ge\sigma-\delta/2$.  The semigroup bound~\eqref{eq:decaybound} gives
$\|h(t)\|_{L^2(\pi)}\le\|h(0)\|_{L^2(\pi)}\exp(-(\sigma-\delta/2)t+\uppi/2)$
for any initial $h(0)=q(0)-1\in\cY$.
From the estimates in Section \ref{sec:warmup} for warm-start equivalence, we have
$c_1^{-1}H(t)\le\|h(t)\|_{L^2(\pi)}^2\le c_2^{-1}H(t)$, hence
$H(t)\le C\exp(-(\sigma-\delta)t)$ for a constant $C$ depending only on $c_1,c_2$
and the initial entropy.
\end{proof}

\section{Asymptotic relaxation enhancing flows}\label{section:examples}


As we have seen, building a single relaxation enhancing flow is often difficult and places strong restrictions on the velocity field. Motivated by the asymptotic rate result for flows that do have $H^1$ eigenfunctions, we now introduce a more flexible notion. We only consider flows with finite energy, meaning 
\[
\int |v_n|^2\pi dx < \infty,\quad \forall n.
\]

\begin{definition}
A sequence of flows $v_n$ with finite energies is called asymptotic relaxation enhancing if $r(v_n)\to\infty$.
\end{definition}

Clearly, one has
\begin{theorem}\label{thm:entropydecay_are}
Suppose that $q_0$ satisfies Assumption \ref{ass:warmstart} with $q_0-1\in \cY$. Let $\{v_n\}$ be a family of asymptotically relaxation enhancing with finite energy, then for any $B>0$, we can find $v_{n}$ in this family and $A_{n}>0$ such that for any $A>A_{n}$, the corresponding $\mathcal{H}(\rho(t)|\pi)$ satisfies
\[
\mathcal{H}(\rho(t)|\pi)\le C\exp(-Bt), \quad \forall t\ge 0.
\]
\end{theorem}
\begin{proof}
By the asymptotic rate result, Theorem \ref{thm:REA} in Section~\ref{section:speedup}, for any flow $v$ and any $\delta>0$, there exists $A_0$ such that for $A>A_0$, $H(t)\le C\exp(-(r(v)-\delta)t)$. Since $r(v_n)\to\infty$, choose $n$ with $r(v_n)>B+1$ and take $\delta=1$. Then for large $A$, $H(t)\le C\exp(-Bt)$.
\end{proof}

As a corollary, one has
\begin{corollary}
Suppose the invariant measure $\pi$ satisfies a log-Sobolev inequality, then under the conditions in Theorem \ref{thm:entropydecay_are}, one has
\begin{gather}
W_1(\rho(t),\pi)\le W_2(\rho(t),\pi)\le C_1\exp(-Bt/2).
\end{gather}
\end{corollary}
\begin{proof}
From Theorem \ref{thm:entropydecay_are}, we have $H(t) = \mathcal{H}(\rho(t)|\pi) \le C \exp(-Bt)$. Since $\pi$ satisfies a log-Sobolev inequality with constant $\lambda$, the Talagrand transportation inequality (stated in Section~\ref{section:setup}) gives $\mathcal{W}_2(\rho(t),\pi) \le \sqrt{\frac{2}{\lambda} \mathcal{H}(\rho(t)|\pi)}$. Combining these yields $\mathcal{W}_2(\rho(t),\pi) \le \sqrt{\frac{2C}{\lambda}} \exp(-Bt/2)$. The bound for $\mathcal{W}_1$ follows from the general inequality $\mathcal{W}_1 \le \mathcal{W}_2$, which holds for any measures with finite second moments by Jensen's inequality.
\end{proof}

Below, we explore some possible ways to obtain the asymptotic relaxation enhancing flows.

\subsection{Approximating relaxation enhancing flows}

We consider the case that $v_n$ converges to some relaxation enhancing flow $v$. We show that such a sequence is asymptotic relaxation enhancing.
We will have the following claim.
\begin{proposition}
Let $v$ be a relaxation enhancing flow, i.e., $\dot{\mathcal{I}}=\emptyset$. Suppose that $\{v_n\}$ is a sequence of flows with finite energies that converges to $v$ weakly in $L^2(\pi)$ and that
\begin{equation}\label{eq:bounded_advection}
\|v_n\cdot\nabla\phi\|_{L^2(\pi)}\le C\|\phi\|_{H^1(\pi)}
\end{equation}
Then,
\[
\lim_{n\to\infty}r(v_n)=\infty.
\]
\end{proposition}

\begin{proof}
We argue by contradiction. Suppose the conclusion does not hold. Then there exists a subsequence (still denoted by $v_n$) and a constant $M<\infty$ such that $r(v_n)\le M$ for all $n$. For each $v_n$, choose an eigenpair $(i\mu_n,\phi_n)$ with $\phi_n \in\dot{\mathcal{I}}(v_n)$, $\|\phi_n\|_{L^2(\pi)}=1$ and
\[
\|\nabla\phi_n\|_{L^2(\pi)}^2 \le r(v_n) + \frac{1}{n} \le M + 1.
\]
Thus $\{\phi_n\}$ is bounded in $H^1(\pi)$. By the compact embedding $H^1(\pi)\hookrightarrow L^2(\pi)$ (which holds under Assumption~\ref{ass:Ureg}), we may extract a subsequence (again denoted $\phi_n$) that converges weakly in $H^1(\pi)$ and strongly in $L^2(\pi)$ to some $\phi_\infty\in H^1(\pi)$ with $\|\phi_\infty\|_{L^2(\pi)}=1$.

Next, we bound the eigenvalues. Using the eigenfunction equation $v_n\cdot\nabla\phi_n = i\mu_n\phi_n$ and the hypothesis \eqref{eq:bounded_advection},
\[
|\mu_n| = \|\mu_n\phi_n\|_{L^2(\pi)} = \|v_n\cdot\nabla\phi_n\|_{L^2(\pi)} \le C\|\phi_n\|_{H^1(\pi)} \le C(M+2)^{1/2}.
\]
Hence $\{\mu_n\}$ is bounded in $\mathbb{R}$. Passing to a further subsequence, we may assume $\mu_n \to \mu$ for some $\mu\in\mathbb{R}$.

We now pass to the limit in the weak formulation of the eigenfunction equation. For any test function $\psi\in C_c^2(\Omega)$,
\[
i\mu_n\int_\Omega \phi_n\overline{\psi}\,\pi(dx) = \int_\Omega (v_n\cdot\nabla\phi_n)\overline{\psi}\,\pi(dx) = -\int_\Omega \phi_n\, v_n\cdot\nabla\overline{\psi}\,\pi(dx),
\]
where we used the weighted divergence-free condition $\nabla\cdot(\pi v_n)=0$ to integrate by parts. Since $\phi_n\to\phi_\infty$ strongly in $L^2(\pi)$ and $v_n\to v$ weakly in $L^2(\pi)$, we have $v_n\cdot\nabla\overline{\psi} \to v\cdot\nabla\overline{\psi}$ weakly in $L^2(\pi)$. The product of a strongly convergent sequence and a weakly convergent sequence converges in the sense of distributions, so the right-hand side converges to $-\int \phi_\infty\, v\cdot\nabla\overline{\psi}\,\pi(dx)$. Taking the limit $n\to\infty$, we obtain
\[
i\mu\int_\Omega \phi_\infty\overline{\psi}\,\pi(dx) = -\int_\Omega \phi_\infty\, v\cdot\nabla\overline{\psi}\,\pi(dx) = \int_\Omega (v\cdot\nabla\phi_\infty)\overline{\psi}\,\pi(dx),
\]
where the last equality follows from another integration by parts (valid since $\phi_\infty\in H^1(\pi)$ and $v$ is smooth). Since $\psi$ is arbitrary, we conclude that $v\cdot\nabla\phi_\infty = i\mu\phi_\infty$ in the weak sense, and thus $\phi_\infty\in\dot{\mathcal{I}}(v)$. But $v$ is relaxation enhancing, so $\dot{\mathcal{I}}(v)=\varnothing$. This contradiction completes the proof.
\end{proof}

Approximation of a relaxation enhancing flow is a possible way to get asymptotic relaxation enhancing sequences on compact manifolds. We next give a different construction on the torus that uses small scales.

\subsection{Scaling to small scales}


Besides approximating relaxation enhancing flows as in the previous section, one can also take advantage of small scales. Intuitively, if the velocity field changes rapidly in space, it creates small scale features and helps mixing. One expects that such small scales can be used to build asymptotic relaxation enhancing flows. In this section, we construct an explicit example by rescaling.

We work on the torus $\mathbb{T}^2=[0,1]^2$ with the flat measure $\pi=1$, and consider the space of all mean-zero functions, i.e., $\cY=\cX$ (see Section~\ref{section:speedup}, formula \eqref{eq:mzspace}\eqref{eq:mzsspace} for the definitions).  Consider the streamfunction
\begin{gather}
\Psi_n=\sin(2\uppi x n)\cos(2\uppi y n),
\end{gather}
where again $\uppi=3.14\cdots$ is the circumference ratio.
Let $v_n$ be the cellular flow
\begin{gather}
v_n=\nabla^{\perp}\Psi_n.
\end{gather}
For any function $f$ that is constant on the streamlines $S_C=\{(x,y):\Psi_n(x,y)=C\}$ with $C\in[-1,1]$, we have $v_n\cdot\nabla f=0$. Hence $f$ is an eigenfunction with eigenvalue zero, and the set $\dot{\mathcal{I}}(v_n)$ defined in \eqref{eigenspace} is nonempty. The following proposition shows that the rate $r(v_n)$ grows like $n^2$.
\begin{proposition}\label{pro:asymCell}
    For $v_n=\nabla^{\perp}\Psi_n$, we have 
    \begin{gather}
    r(v_n)\geq n^2 \delta,
    \end{gather}
     for some $\delta>0$ independent of $n$.
\end{proposition}

\begin{remark}
    By applying Proposition~\ref{prop:from1to2} to the cellular flows $v_n$, we obtain a family of asymptotic relaxation enhancing flows for any smooth target density $\pi dxdy$ on $\mathbb{T}^2$.
\end{remark}

To prove the proposition, we first establish a elementary fact. In the following, we will introduce the small cell
\begin{equation}\label{eq:cellij}
C_{i,j}=\left[\frac{i-1}{2n},\frac{i}{2n}\right]\times \left[\frac{j-1}{2n},\frac{j}{2n}\right].
\end{equation}

\begin{lemma}
For $\sigma<1$, let $U_{\sigma}=\{(x, y): |\Psi_n|> 1-\sigma\}$ and $K_{\sigma}=\T^2\setminus U_{\sigma}$. Suppose $\varphi \in H^1$ is real-valued and satisfies
\[
v\cdot\nabla\varphi=0.
\]
Then $\varphi$ is continuous on $K_{\sigma}$, and on the boundaries of all cells $C_{i,j}$ it takes a common constant value.
\end{lemma}

\begin{proof}
The set $K_{\sigma}$ is obtained from $\T^2$ by removing small neighbourhoods of the cell centers, where $|\Psi_n| > 1-\sigma$. Thus $K_{\sigma}$ contains the interior of each cell up to distance $\sigma$ from the centre, as well as all cell boundaries, namely the curves where $\Psi_n = 0$.

On the interior of a cell $C_{i,j}$, the level sets of $\Psi_n$ are closed streamlines of $v_n = \nabla^\perp \Psi_n$. Because $v_n \cdot \nabla \varphi = 0$, the function $\varphi$ is constant along these streamlines. Consequently, on the part of $C_{i,j}$ that lies in $K_{\sigma}$ and away from the boundary we can write $\varphi(x,y) = g_{ij}(\Psi_n(x,y))$ for some function $g_{ij}$ defined on the open interval of values taken by $\Psi_n$ on that region. Since $\Psi_n$ has constant sign on each cell, this open interval is either $(0,1-\sigma)$ or $(-1+\sigma,0)$; we denote it by $I_{ij}$.

To study the regularity of $g_{ij}$, fix a compact subinterval $[a,b] \subset I_{ij}$. On the set $R = \{(x,y) \in C_{i,j} : a \le \Psi_n \le b\}$, which lies entirely in the interior of the cell, the gradient $|\nabla\Psi_n|$ is bounded away from zero. Introduce coordinates $(u,s)$ where $u = \Psi_n(x,y)$ and $s$ is the arclength along the level curve $\{\Psi_n = u\}$. The map $(x,y) \mapsto (u,s)$ is a smooth change of variables on $R$, and the area element satisfies $dxdy = |\nabla\Psi_n|^{-1} du\,ds$. Since $|\nabla\varphi| = |g_{ij}'(u)| \, |\nabla\Psi_n|$, we obtain
\[
\int_R |\nabla\varphi|^2 \,dxdy = \int_a^b |g_{ij}'(u)|^2 J(u) \,du,
\]
where $J(u) = \int_{\Psi_n = u} |\nabla\Psi_n| \,ds$. The function $J$ is smooth and positive on $(-1,1)$; the integral is well defined even at $u = 0$ because the corners of the cell boundaries are isolated points and do not affect the arclength measure. In particular, $J(u) \ge c > 0$ on $[a,b]$. Because $\varphi \in H^1(K_{\sigma})$, the left hand side is finite, forcing $\int_a^b |g_{ij}'(u)|^2 \,du < \infty$. Thus $g_{ij} \in H^1([a,b])$, and by the one dimensional Sobolev embedding, $g_{ij}$ is continuous on $[a,b]$. Since $[a,b]$ was arbitrary, $g_{ij}$ is continuous on the whole open interval $I_{ij}$.

Now consider a point $P_0$ on an open edge of the cell boundary, away from the corners. At $P_0$, $\Psi_n = 0$ and $\nabla\Psi_n \neq 0$. Choose a small neighbourhood $\mathcal{N}$ of $P_0$ contained in $K_{\sigma}$ on which $\nabla\Psi_n \neq 0$. On $\mathcal{N}$ we can introduce local coordinates $(u,t)$ where $u = \Psi_n(x,y)$ and $t$ is a coordinate along the level sets, for instance the arclength along $\{\Psi_n = 0\}$ extended to nearby level sets. The map $(x,y) \mapsto (u,t)$ is a smooth diffeomorphism onto a rectangle $(-\delta,\delta) \times I$. In these coordinates, $v_n = \nabla^\perp\Psi_n$ is tangent to the level sets $\{u = \text{const}\}$, so the equation $v_n \cdot \nabla \varphi = 0$ becomes $\partial_t \varphi = 0$. Hence $\varphi$ depends only on $u$: there exists a function $h$ on $(-\delta,\delta)$ such that $\varphi(x,y) = h(u)$ for all $(x,y) \in \mathcal{N}$. The $H^1$ regularity of $\varphi$ implies $h \in H^1((-\delta,\delta))$, and therefore $h$ is continuous on $(-\delta,\delta)$.

For $u \in (0,\delta)$ (respectively $(-\delta,0)$), the points in $\mathcal{N}$ with $\Psi_n = u$ lie in one of the adjacent cells, and there we already have $\varphi = g_{ij}(u)$ for the corresponding $g_{ij}$. Consequently, $h(u) = g_{ij}(u)$ for $u \neq 0$. This shows that each $g_{ij}$ admits a continuous extension to $u = 0$ given by $g_{ij}(0) = h(0)$. Moreover, the limits of $g_{ij}$ from the two sides of the edge coincide, so $\varphi$ is continuous across the open edge.

We have now shown that each $g_{ij}$ is continuous on the closed interval obtained by adding $0$ to $I_{ij}$. Because $\Psi_n$ is continuous and $\varphi = g_{ij}(\Psi_n)$ holds on $C_{i,j} \cap K_{\sigma}$, $\varphi$ can be extended to a continuous function on each closed cell intersected with $K_{\sigma}$. The continuity across open edges guarantees that $\varphi$ is continuous on the union of any two adjacent cells.  
It remains to check the corners. Consider a corner point where four cells meet. The four open edges that meet at the corner each carry a constant value of $\varphi$, given by the limits $g_{ij}(0)$ from the corresponding cells. Because $\varphi$ is continuous across each open edge, the values on the two edges that border a given cell must be equal; otherwise the limit from that cell would disagree with the value on one of its edges. Repeating this argument around the corner shows that all four limits $g_{ij}(0)$ are equal to a common value, which we denote by $a_0$. We can therefore extend $\varphi$ to the corner by setting $\varphi = a_0$ there.  Thus $\varphi$ is continuous on the entire set $K_{\sigma}$ and takes a single constant value $a$ on all cell boundaries.
\end{proof}

Next, we prove Proposition \ref{pro:asymCell}.
\begin{proof}[Proof of Proposition \ref{pro:asymCell}]
Let $\phi \in \dot{\mathcal{I}}(v_n)$ be a normalized eigenfunction, i.e., $\|\phi\|_{L^2}=1$ and $v_n\cdot\nabla\phi = i\mu\phi$ for some $\mu\in\mathbb{R}$. Write $\phi = \Phi e^{i\theta}$ with $\Phi = |\phi| \ge 0$ and $\theta$ real-valued (the phase $\theta$ may be multi-valued, but its gradient $\nabla\theta$ is well-defined almost everywhere). We have
\[
|\nabla\phi|^2 = |\nabla\Phi|^2 + \Phi^2 |\nabla\theta|^2,\qquad \int_{\T^2} \Phi^2\,dxdy = 1.
\]

Since $\phi \in H^1(\T^2)$ and the absolute value function is Lipschitz, the chain rule for Sobolev spaces gives $\Phi = |\phi| \in H^1(\T^2)$, and $|\nabla\Phi| \le |\nabla\phi|$ almost everywhere.

We now show that $\Phi$ is constant along the streamlines of $v_n$. Using $\Phi^2 = |\phi|^2$ and the eigenvalue equation,
\[
v_n\cdot\nabla(\Phi^2) = 2\operatorname{Re}\bigl(\bar{\phi}\, v_n\cdot\nabla\phi\bigr)
= 2\operatorname{Re}\bigl(i\mu|\phi|^2\bigr) = 0.
\]
Thus $\Phi^2$, and hence $\Phi$, is a first integral of the flow. By the lemma proved above, $\Phi$ is continuous on the cell boundaries (the curves where $\Psi_n = 0$ or $\Psi_n = \pm 1$) and takes a single constant value on all of them. Denote this common value by $a_\phi$.

We split the argument into two main cases.

\medskip
\noindent \textbf{Case 1: $|a_\phi - 1| \ge \delta_1$ for some fixed small $\delta_1>0$.}
The torus $\T^2$ is partitioned into $4n^2$ small cells
\[
C_{i,j} = \Bigl[ \frac{i-1}{2n}, \frac{i}{2n} \Bigr] \times \Bigl[ \frac{j-1}{2n}, \frac{j}{2n} \Bigr], \qquad i,j = 1,\dots,2n.
\]
On each cell $C_{i,j}$, the function $\Phi - a_\phi$ vanishes on the boundary because $\Phi = a_\phi$ on all cell boundaries.

We rescale each cell to the unit square $[0,1]^2$. Define the change of variables
\[
x = \frac{i-1+x'}{2n}, \quad y = \frac{j-1+y'}{2n}, \qquad (x',y') \in [0,1]^2,
\]
and set $G_{ij}(x',y') = \Phi(x,y) - a_\phi$. The Jacobian is $\frac{\partial(x,y)}{\partial(x',y')} = \frac{1}{4n^2}$, and the gradient scales as $\nabla_{x,y} = 2n \nabla_{x',y'}$. Since $\Phi \in H^1$ and $\Phi = a_\phi$ on the boundary, $G_{ij} \in H_0^1((0,1)^2)$.

We compute the $L^2$ norms on $C_{i,j}$:
\[
\int_{C_{i,j}} |\nabla\Phi|^2 \,dxdy
= \int_{[0,1]^2} \bigl( (2n)^2 |\nabla G_{ij}|^2 \bigr) \cdot \frac{1}{4n^2} \,dx'dy'
= \int_{[0,1]^2} |\nabla G_{ij}|^2 \,dx'dy',
\]
and
\[
\int_{C_{i,j}} |\Phi - a_\phi|^2 \,dxdy
= \int_{[0,1]^2} |G_{ij}|^2 \cdot \frac{1}{4n^2} \,dx'dy'
= \frac{1}{4n^2} \int_{[0,1]^2} |G_{ij}|^2 \,dx'dy'.
\]

The Poincar\'{e} inequality on the unit square for functions in $H_0^1((0,1)^2)$ states that there exists a constant $\kappa_0 > 0$ such that
\[
\int_{[0,1]^2} |\nabla G_{ij}|^2 \,dx'dy' \ge \kappa_0 \int_{[0,1]^2} |G_{ij}|^2 \,dx'dy'.
\]
Substituting the scaled integrals yields
\[
\int_{C_{i,j}} |\nabla\Phi|^2 \,dxdy \ge 4n^2 \kappa_0 \int_{C_{i,j}} |\Phi - a_\phi|^2 \,dxdy.
\]
If $\Phi$ is constant on $C_{i,j}$, both sides are zero and the inequality holds trivially. Summing over all $4n^2$ cells gives
\[
\int_{\T^2} |\nabla\Phi|^2 \,dxdy \ge 4n^2 \kappa_0 \int_{\T^2} |\Phi - a_\phi|^2 \,dxdy.
\]
Now, using $\int_{\T^2} \Phi^2 = 1$ and the assumption $|a_\phi - 1| \ge \delta_1$, we estimate $\int_{\T^2} |\Phi - a_\phi|^2 \,dxdy$.
We treat the two possibilities separately.

\emph{If $a_\phi \le 1 - \delta_1$:} For any $\epsilon > 0$, Young's inequality gives $2a_\phi\Phi \le \epsilon \Phi^2 + \frac{a_\phi^2}{\epsilon}$.
Hence
\[
|\Phi - a_\phi|^2 = \Phi^2 - 2a_\phi\Phi + a_\phi^2 \ge \Phi^2 - \epsilon \Phi^2 - \frac{a_\phi^2}{\epsilon} + a_\phi^2 = (1-\epsilon)\Phi^2 + \Bigl(1 - \frac{1}{\epsilon}\Bigr)a_\phi^2.
\]
Choose $\epsilon = 1 - \delta_1$ (note $\epsilon > 0$ since $\delta_1 < 1$), one thus has
\[
\int_{\T^2} |\Phi - a_\phi|^2 \,dxdy \ge \delta_1 - \frac{\delta_1}{1-\delta_1} a_\phi^2
\ge \delta_1 - \frac{\delta_1}{1-\delta_1} (1-\delta_1)^2 = \delta_1 - \delta_1(1-\delta_1) = \delta_1^2.
\]

\emph{If $a_\phi \ge 1 + \delta_1$:} We again apply Young's inequality, this time with $\epsilon = 1 + \delta_1$, and have
\[
|\Phi - a_\phi|^2 = \Phi^2 - 2a_\phi\Phi + a_\phi^2 \ge \Phi^2 - \epsilon \Phi^2 - \frac{a_\phi^2}{\epsilon} + a_\phi^2 = (1-\epsilon)\Phi^2 + \Bigl(1 - \frac{1}{\epsilon}\Bigr)a_\phi^2.
\]
With $\epsilon = 1 + \delta_1$, we get 
\[
\int_{\T^2} |\Phi - a_\phi|^2 \,dxdy \ge -\delta_1 + \frac{\delta_1}{1+\delta_1} a_\phi^2
 \ge -\delta_1 + \frac{\delta_1}{1+\delta_1} (1+\delta_1)^2 = \delta_1^2.
\]

In both situations, the desired estimate holds.
Since $|\nabla\phi|^2 \ge |\nabla\Phi|^2$, we conclude $\|\nabla\phi\|_{L^2}^2 \ge 4\kappa_0 \delta_1^2 n^2$.

\medskip
\noindent \textbf{Case 2: $|a_\phi - 1| < \delta_1$.}
We first prove that $\mu = 0$. Substituting $\phi = \Phi e^{i\theta}$ into the eigenvalue equation $v_n\cdot\nabla\phi = i\mu\phi$ gives
\[
\Phi\, v_n\cdot\nabla\theta = \mu\Phi.
\]
Although the phase $\theta$ may be multi-valued, the function $e^{i\theta}$ is single-valued (since $\phi$ is single-valued and $\Phi = |\phi|$), and the gradient $\nabla\theta$ is well-defined almost everywhere on the set where $\Phi > 0$. By the lemma, $\Phi$ is continuous on $K_\sigma$ (the region outside a small neighborhood of the cell centers) and equals the constant $a_\phi$ on all cell boundaries. Since $|a_\phi - 1| < \delta_1$, taking $\delta_1$ small ensures $a_\phi$ is close to $1$. By continuity, there exists $\sigma_0 > 0$ such that $\Phi(x,y) \ge \frac{1}{2}$ for all $(x,y) \in K_\sigma$ whenever $0 < \sigma < \sigma_0$. On this set we may divide by $\Phi$ to obtain
\[
v_n\cdot\nabla\theta = \mu.
\]

Now fix a cell $C_{i,j}$ and a streamline $\gamma$ of $\Psi_n$ lying in $C_{i,j} \cap K_\sigma$. The streamline is a closed curve, and $v_n = \nabla^\perp\Psi_n$ is tangent to it with speed $|v_n| = |\nabla\Psi_n|$. Parameterizing $\gamma$ by arclength $s$, the derivative of $\theta$ along $\gamma$ satisfies
\[
\frac{d\theta}{ds} = \frac{\mu}{|\nabla\Psi_n|}.
\]
Integrating along the entire closed streamline gives
\[
\oint_\gamma d\theta = \mu \int_0^L \frac{ds}{|\nabla\Psi_n|},
\]
where $L$ is the length of $\gamma$. Since $e^{i\theta}$ is single-valued, the integral $\oint_\gamma d\theta$ must be an integer multiple of $2\uppi$.

Now consider the quantity $I(c) = \int_{\Psi_n = c} \frac{ds}{|\nabla\Psi_n|}$, where $c$ is the constant value of $\Psi_n$ on the streamline. As the streamline approaches the cell center (i.e., as $c \to \pm 1$), the gradient $|\nabla\Psi_n|$ tends to zero, causing $I(c)$ to blow up. On the other hand, near the cell boundaries (where $c=0$ or $c=\pm 1$ between cells), $|\nabla\Psi_n|$ is bounded away from zero and $I(c)$ is finite. Hence $I(c)$ is not constant as a function of $c$. If $\mu \neq 0$, then $\mu I(c)$ would have to be an integer multiple of $2\pi$ for every $c$, which is impossible since $I(c)$ varies continuously and is not constant. Therefore we must have $\mu = 0$.

Consequently, $\theta$ is also a first integral of the flow.

Applying the same regularity argument as for $\Phi$, we conclude that $\theta$ is continuous on each $K_{\sigma}\cap \{\Phi>0\}$ and constant on cell boundaries. After a global phase shift, we may assume $\theta = 0$ on all cell boundaries. The mean-zero condition $\int_{\T^2} \phi \,dxdy = 0$ implies that both the real and imaginary parts vanish. In particular, taking the real part gives
\[
\int_{\T^2} \Phi \cos\theta \,dxdy = 0.
\]
(The imaginary part yields $\int \Phi \sin\theta dxdy = 0$, which is not needed for the estimates.)

Fix small positive numbers $\sigma_1, \delta_2, \delta_3$ (to be chosen appropriately). We consider two subcases.

\medskip
\noindent \textit{Subcase 2.1:} There are at least $\lceil \delta_2 4n^2 \rceil$ cells for which $\Phi$ takes a value outside the interval $[1-2\delta_1, 1+2\delta_1]$ at some point in $C_{i,j} \cap K_{\sigma_1}$.

Fix such a cell $C_{i,j}$. On this cell, $\Phi$ is constant along streamlines, so we can write $\Phi = g(\Psi_n)$ for some function $g$ defined on the range of $\Psi_n$ within the cell. Since $\Phi \in H^1$ and $|\nabla\Psi_n|$ is bounded away from zero on $K_{\sigma_1}$, $g$ belongs to $H^1$ on the corresponding interval, and we denote its weak derivative by $g'$. The assumption that $\Phi$ reaches a value at distance at least $2\delta_1$ from $1$ implies that $g$ must vary by at least $\delta_1$ over the interval of $\Psi_n$ values that lie in $K_{\sigma_1}$. Precisely, because $|a_\phi - 1| < \delta_1$, if $\Phi$ goes below $1-2\delta_1$ or above $1+2\delta_1$, then $|g(u) - a_\phi| \ge \delta_1$ for some $u \in [-1+\sigma_1, 1-\sigma_1]$. Hence
\[
\int_{-1+\sigma_1}^{1-\sigma_1} |g'(u)| \,du \ge \delta_1.
\]

Now we relate this variation to the $L^2$ norm of $\nabla\Phi$. In $K_{\sigma_1} \cap C_{i,j}$, we have $|\nabla\Phi| = |g'(\Psi_n)| \, |\nabla\Psi_n|$. Introduce coordinates $(u,s)$ where $u = \Psi_n(x,y)$ and $s$ is the arclength along the streamline $\Psi_n = u$. The change of variables has Jacobian $1/|\nabla\Psi_n|$, so
\[
dxdy = \frac{du \, ds}{|\nabla\Psi_n|}.
\]
Then
\[
\int_{K_{\sigma_1} \cap C_{i,j}} |\nabla\Phi|^2 \,dxdy 
= \int_{-1+\sigma_1}^{1-\sigma_1} |g'(u)|^2 \Bigl( \int_{\Psi_n=u} |\nabla\Psi_n| \,ds \Bigr) du.
\]
Define $J_n(u) = \int_{\Psi_n=u} |\nabla\Psi_n| \,ds$. By scaling, $J_n(u)$ is independent of $n$: indeed, the level set $\Psi_n = u$ is the $1/n$-scaled version of $\Psi_1 = u$, so $|\nabla\Psi_n| = n |\nabla\Psi_1|$ and $ds_n = (1/n) ds_1$, giving $J_n(u) = J_1(u)$. Thus $J_n(u)$ is a smooth, positive function of $u$ on $(-1,1)$. On the compact interval $[-1+\sigma_1, 1-\sigma_1]$, $J_n$ is continuous and positive, hence attains a positive minimum. Therefore there exist constants $c_1, c_2 > 0$ independent of $n$ such that $c_1 \le J_n(u) \le c_2$ for all $u \in [-1+\sigma_1, 1-\sigma_1]$. 

Applying the Cauchy-Schwarz inequality,
\[
\delta_1 \le \int_{-1+\sigma_1}^{1-\sigma_1} |g'(u)| \,du 
\le \Bigl( \int_{-1+\sigma_1}^{1-\sigma_1} |g'(u)|^2 J_n(u) \,du \Bigr)^{1/2}
\Bigl( \int_{-1+\sigma_1}^{1-\sigma_1} \frac{1}{J_n(u)} \,du \Bigr)^{1/2}.
\]
Since $J_n(u) \ge c_1$, the second factor is at most $C := \bigl( \frac{2}{c_1} \bigr)^{1/2}$. Squaring and rearranging gives
\[
\int_{-1+\sigma_1}^{1-\sigma_1} |g'(u)|^2 J_n(u) \,du \ge \frac{\delta_1^2}{C^2}.
\]
Therefore, for each such cell $C_{i,j}$, we obtain
\[
\int_{K_{\sigma_1} \cap C_{i,j}} |\nabla\Phi|^2 \,dxdy \ge c_1'' := \frac{\delta_1^2}{C^2} > 0,
\]
where $c_1''$ is independent of $n$.

Summing over the at least $\lceil \delta_2 4n^2 \rceil$ cells in this subcase, we get
\[
\int_{\T^2} |\nabla\Phi|^2 \ge c_1'' \delta_2 4n^2.
\]
Since $|\nabla\phi|^2 \ge |\nabla\Phi|^2$, the desired $O(n^2)$ lower bound follows.

\medskip
\noindent \textit{Subcase 2.2:} The situation in \textit{Subcase 2.1} does not happen. Then for at least $\lfloor (1-\delta_2)4n^2 \rfloor$ cells, the inequality $|\Phi(x,y) - 1| \le 2\delta_1$ holds for every point $(x,y)$ in $C_{i,j} \cap K_{\sigma_1}$. Let $\mathcal{C}_1$ be the collection of all such cells.

Define $S_1 = \{(x,y) : |\Phi(x,y) - 1| \le 2\delta_1\}$. We estimate the area of $S_1$. The total area of the torus is $1$. Each cell $C_{i,j}$ has area $1/(4n^2)$. The number of cells in $\mathcal{C}_1$ is at least $\lfloor (1-\delta_2)4n^2 \rfloor$, and since $\lfloor x \rfloor \ge x - 1$, this number is at least $(1-\delta_2)4n^2 - 1$. Therefore, the total area of the cells in $\mathcal{C}_1$ is at least
\[
\bigl( (1-\delta_2)4n^2 - 1 \bigr) \cdot \frac{1}{4n^2} = 1 - \delta_2 - \frac{1}{4n^2}.
\]
On each cell in $\mathcal{C}_1$, the inequality $|\Phi - 1| \le 2\delta_1$ holds on $C_{i,j} \cap K_{\sigma_1}$. The only points in these cells where the inequality may fail are those in $U_{\sigma_1}$, the small neighborhoods of the cell centers. The total area of $U_{\sigma_1}$ is $|U_{\sigma_1}|$. Cells not in $\mathcal{C}_1$ have total area at most $\delta_2 + \frac{1}{4n^2}$. Consequently, the area of $S_1$ satisfies
\[
|S_1| \ge 1 - \delta_2 - \frac{1}{4n^2} - |U_{\sigma_1}|.
\]

We now bound the integral of $\Phi$ over the set where $\Phi$ is large. Since $\|\phi\|_{L^2} = 1$, we have $\int_{\T^2} \Phi^2 = 1$. Splitting the integral gives
\[
1 = \int_{S_1} \Phi^2 \,dxdy + \int_{S_1^c} \Phi^2 \,dxdy.
\]
On $S_1$, we have $\Phi \ge 1-2\delta_1$, so $\Phi^2 \ge (1-2\delta_1)^2$ pointwise. Hence $\int_{S_1} \Phi^2 \ge (1-2\delta_1)^2 |S_1|$. On the complement $S_1^c$, we restrict to the subset where $\Phi \ge 1+2\delta_1$; on this subset, $\Phi^2 \ge (1+2\delta_1)\Phi$. Therefore,
\[
1 \ge (1-2\delta_1)^2 |S_1| + (1+2\delta_1) \int_{\Phi \ge 1+2\delta_1} \Phi \,dxdy.
\]
Rearranging, we obtain
\[
\int_{\Phi \ge 1+2\delta_1} \Phi \,dxdy \le \frac{1 - (1-2\delta_1)^2 |S_1|}{1+2\delta_1}.
\]

Let $S_2 = \{(x,y) : |\theta(x,y)| \le \delta_3\}$. Now we use the mean-zero condition $\int_{\T^2} \Phi \cos\theta \,dxdy = 0$. We split the integral into three parts:
\[
0 = \int_{S_1 \cap S_2} \Phi \cos\theta + \int_{S_1 \setminus S_2} \Phi \cos\theta + \int_{S_1^c} \Phi \cos\theta.
\]
On $S_1$, we have $|\Phi - 1| \le 2\delta_1$, so in particular $\Phi \le 1+2\delta_1$ and $\Phi \ge 1-2\delta_1$. On $S_1 \cap S_2$, we also have $|\theta| \le \delta_3$, so $\cos\theta \ge \cos\delta_3$. Therefore,
\[
\int_{S_1 \cap S_2} \Phi \cos\theta \ge (1-2\delta_1) \cos\delta_3 \, |S_1 \cap S_2|.
\]
On $S_1 \setminus S_2$, we have $\Phi \le 1+2\delta_1$ and $\cos\theta \ge -1$, hence
\[
\int_{S_1 \setminus S_2} \Phi \cos\theta \ge -(1+2\delta_1) \, |S_1 \setminus S_2|.
\]
On $S_1^c$, we use the triangle inequality:
\[
\left| \int_{S_1^c} \Phi \cos\theta \right| \le \int_{S_1^c} \Phi \,dxdy,
\]
since $\Phi=|\phi|\ge 0$. We bound the integral of $\Phi$ on $S_1^c$ by splitting it into the region where $\Phi \le 1-2\delta_1$ and the region where $\Phi \ge 1+2\delta_1$. On the first region, $\Phi \le 1-2\delta_1$, and its area is at most $|S_1^c| = 1 - |S_1|$. On the second region, we already have the estimate from the $L^2$ normalization:
\[
\int_{\Phi \ge 1+2\delta_1} \Phi \le \frac{1 - (1-2\delta_1)^2 |S_1|}{1+2\delta_1}.
\]
Therefore,
\[
\int_{S_1^c} \Phi \le (1-2\delta_1)(1 - |S_1|) + \frac{1 - (1-2\delta_1)^2 |S_1|}{1+2\delta_1}.
\]

Putting these together, the mean-zero condition gives
\[
(1-2\delta_1) \cos\delta_3 \, |S_1 \cap S_2| - (1+2\delta_1) |S_1 \setminus S_2| - \int_{S_1^c} \Phi \le 0.
\]
Note that $|S_1 \setminus S_2| = |S_1| - |S_1 \cap S_2|$. Substituting and rearranging, we obtain
\[
\bigl[ (1-2\delta_1) \cos\delta_3 + (1+2\delta_1) \bigr] |S_1 \cap S_2| \le (1+2\delta_1) |S_1| + \int_{S_1^c} \Phi.
\]

Now, using the estimate for $\int_{S_1^c} \Phi$ and the fact that $|S_1| \le 1$, $|S_1^c| = 1 - |S_1|$, we get
\[
|S_1 \cap S_2| \le \frac{(1+2\delta_1)|S_1| + (1-2\delta_1)(1-|S_1|) + \frac{1 - (1-2\delta_1)^2 |S_1|}{1+2\delta_1}}{(1-2\delta_1)\cos\delta_3 + (1+2\delta_1)}.
\]

Recall that $|S_1| \ge 1 - \delta_2 - \frac{1}{4n^2} - |U_{\sigma_1}|$. Taking $\sigma_1, \delta_2, \delta_3$ small and $n$ large, we can make $|S_1|$ arbitrarily close to $1$. In the limit $|S_1| \to 1$ and $\delta_1, \delta_3 \to 0$, the right-hand side tends to
\[
\frac{(1)(1) + 0 + \frac{1 - 1}{1}}{1 + 1} = \frac{1}{2}.
\]
Therefore, for sufficiently small choices of the parameters and large $n$, there exists a small $\epsilon_1' > 0$ such that
\[
|S_1 \cap S_2| \le \frac{1}{2} + \epsilon_1'.
\]

Now let $\mathcal{C}_2$ be the set of cells in $\mathcal{C}_1$ for which there exists a point in $C_{i,j} \cap K_{\sigma_1}$ with $|\theta| \ge \delta_3$. We need to bound the number of cells in $\mathcal{C}_1 \setminus \mathcal{C}_2$. For any cell in $\mathcal{C}_1 \setminus \mathcal{C}_2$, we have $|\theta| < \delta_3$ on $C_{i,j} \cap K_{\sigma_1}$, so this entire region is contained in $S_1 \cap S_2$. The area of $C_{i,j} \cap K_{\sigma_1}$ is at least $\frac{1}{4n^2} - |C_{i,j} \cap U_{\sigma_1}|$. Summing over all $N = |\mathcal{C}_1 \setminus \mathcal{C}_2|$ cells, we get
\[
|S_1 \cap S_2| \ge \frac{N}{4n^2} - \sum_{C\in \mathcal{C}_1\setminus\mathcal{C}_2} |C \cap U_{\sigma_1}| \ge \frac{N}{4n^2} - |U_{\sigma_1}|.
\]
Combining this with the upper bound $|S_1 \cap S_2| \le \frac{1}{2} + \epsilon_1'$, we obtain
\[
\frac{N}{4n^2} \le \frac{1}{2} + \epsilon_1' + |U_{\sigma_1}|.
\]
Multiplying by $4n^2$ gives
\[
N \le 2n^2 \bigl(1 + 2\epsilon_1' + 2|U_{\sigma_1}|\bigr) = 2n^2 (1 + \epsilon_1),
\]
where $\epsilon_1 = 2\epsilon_1' + 2|U_{\sigma_1}|$ can be made arbitrarily small by choosing $\sigma_1$ small and $n$ large.

Recall that $|\mathcal{C}_1| \ge (1-\delta_2)4n^2 - 1$. Therefore, the number of cells in $\mathcal{C}_2$ satisfies
\[
|\mathcal{C}_2| = |\mathcal{C}_1| - N \ge (1-\delta_2)4n^2 - 1 - 2n^2(1+\epsilon_1) = \lfloor (1 - \epsilon_1 - 2\delta_2)2n^2 \rfloor.
\]
On each cell in $\mathcal{C}_2$, we have $\Phi \ge 1-2\delta_1$ on $C_{i,j} \cap K_{\sigma_1}$ (since it belongs to $\mathcal{C}_1$) and $\theta$ has a variation of at least $\delta_3$.

For a cell in $\mathcal{C}_2$, a change of variables similar to the one used for $\Phi$ yields
\[
\int_{C_{i,j} \cap K_{\sigma_1}} |\nabla\theta|^2 \,dxdy \ge c_2',
\]
with $c_2' > 0$ independent of $n$. Since $\Phi \ge 1-2\delta_1$ on this region, we have
\[
|\nabla\phi|^2 \ge \Phi^2 |\nabla\theta|^2 \ge (1-2\delta_1)^2 |\nabla\theta|^2.
\]
Summing over all cells in $\mathcal{C}_2$ gives
\[
\int_{\T^2} |\nabla\phi|^2 \ge (1-2\delta_1)^2 c_2' \cdot \lfloor (1 - \epsilon_1 - 2\delta_2)2n^2 \rfloor\ge c_3 n^2
\]
for some constant $c_3 > 0$ independent of $n$.

\medskip
In all cases we have shown that $\|\nabla\phi\|_{L^2}^2 \ge \delta n^2$ for a universal constant $\delta > 0$ (independent of $n$ and the particular eigenfunction $\phi$). Since this holds for every normalized $\phi \in \dot{\mathcal{I}}(v_n)$, we conclude that $r(v_n) \ge \delta n^2$.
\end{proof}

\subsection{The full space}



As we have seen, building a relaxation enhancing flow on a compact manifold like $\T^2$ is already delicate. On the full space $\R^d$, the situation is even more challenging. For a weighted measure $\pi \propto e^{-U}$, any steady flow that is relaxation enhancing with respect to $\pi$ generally has infinite energy $\int_{\R^d} |v|^2 \pi(dx) = \infty$; this can be seen by analyzing the eigenfunctions of $v \cdot \nabla$, which must decay at infinity while supporting a non-zero eigenvalue, forcing the velocity field to grow unboundedly (see the discussion in \cite{zlatovs2010diffusion}). Instead of seeking a single flow with infinite energy, we construct a sequence of compactly supported flows with finite energy that are asymptotically relaxation enhancing.

The idea is to build each $v_n$ so that it agrees with a flow whose eigenfunctions have big $H^1$ norms on a large compact set and vanishes outside a slightly larger ball. For example, we can make use of the small scales of the cellular flows to construct $v_n$ on the compact set. 
For eigenfunctions whose mass is concentrated inside the compact set, the property of the flow forces a large $H^1(\pi)$ norm. For eigenfunctions supported mostly outside, we will heavily rely on the weighted measure. In fact, if a certain 
Lyapunov estimate holds, the  Poincare constant outside the compact set is big, which fulfills the requirement.

To make this precise, we first impose a growth condition on the potential $U$ that guarantees such a Lyapunov estimate.

\begin{assumption}[Lyapunov condition]\label{ass:lyap}
The potential $U$ admits a $C^2$ function $W\ge 1$ with the following property.
There exists a radius $R_0>0$ such that for any increasing sequence of radii $r_n\to\infty$ with $r_n\ge R_0$,
one can find a sequence $\lambda_n\to\infty$ and a nondecreasing sequence $b_n\ge 0$ satisfying
\begin{equation} \label{eq:lyapcond}
\cL_0 W(x) \le -\lambda_n W(x) + b_n \mathbf{1}_{B(0,r_n)}(x) \qquad \text{for all } x\in\mathbb{R}^d,
\end{equation}
where $\cL_0 = -\nabla U\cdot\nabla + \Delta$ is the operator defined in \eqref{eq:L0}. 
\end{assumption}
\begin{remark}
In the literature on Markov processes, a function $W$ with such a drift condition is often called a Lyapunov function.
\end{remark}

We note that this assumption is satisfied whenever the potential $U$ grows sufficiently fast at infinity. Suppose that $U$ is bounded below. Choose $W(x) = M e^{\delta U(x)}$ with $M >0$ such that $W\ge 1$ and a small constant $\delta \in (0,1)$. A direct computation gives
\[
\cL_0 W = \bigl( -(\delta - \delta^2)|\nabla U|^2 + \delta \Delta U \bigr) W.
\]
Define $F(x) = (\delta - \delta^2)|\nabla U(x)|^2 - \delta \Delta U(x)$. Assume that $F(x) \to +\infty$ as $|x|\to\infty$. Then, take $R>0$ such that $F(x)>0$ when $|x|\ge R$ and thus
\[
\lambda_n = \inf_{|x| \ge r_n} F(x)>0.
\]
Consequently, $\lambda_n \to \infty$ and 
\[
\cL_0 W(x) = -F(x) W(x) \le -\lambda_n W(x) \qquad \text{for } |x| \ge r_n.
\]
On the compact ball $B(0,r_n)$, the continuous function $\cL_0 W + \lambda_n W$ is bounded above by some constant $b_n \ge 0$. Therefore we obtain
\[
\cL_0 W \le -\lambda_n W + b_n \mathbf{1}_{B(0,r_n)}.
\]
Note that  Assumption \ref{ass:Ureg} implies that $F(x)\to \infty$ as $|x|\to\infty$ for $\delta=1/2$. Moreover, if  $|\nabla U(x)| \to \infty$ as $|x|\to\infty$ and $|\Delta U(x)| / |\nabla U(x)|^2 \to 0$ as $|x|\to\infty$, then the desired condition $F(x)\to\infty$ holds for any $\delta\in (0, 1)$. 
 In particular, the standard Gaussian measure $U(x) = |x|^2/2$ satisfies these conditions.

The following lemma is a well-known consequence of such a Lyapunov condition (see, e.g., \cite{bakry2008rate}). We include a short proof for completeness.
\begin{lemma}\label{lmm:poincarefaraway}
Assume that $U$ satisfies the Lyapunov condition, and let $\pi \propto e^{-U}$ be the corresponding invariant measure. If $\phi \in H^1(\pi)$ is supported in $\R^d \setminus B(0,r_n)$, 
\[
\int_{\R^d} \phi^2 \,\pi(dx) \le \frac{1}{\lambda_n} \int_{\R^d} |\nabla \phi|^2 \,\pi(dx),
\]
where $\lambda_n$ is the same constant appearing in \eqref{eq:lyapcond}.
\end{lemma}

\begin{proof}
We only have to consider smooth functions supported in $\R^d \setminus B(0,r_n)$. Once the claim holds for smooth functions, the claim can be extended to the $H^1$ functions by density.  Since $\phi$ is supported in $\R^d \setminus B(0,r_n)$, the indicator $\mathbf{1}_{B(0,r_n)}$ vanishes on the support of $\phi$. One has by \eqref{eq:lyapcond} the following pointwisely
\[
\frac{\cL_0 W}{W} \phi^2 \le -\lambda_n \phi^2.
\]
Rearranging and integrating against the invariant measure $\pi$ yields
\begin{equation}
\int_{\R^d} \phi^2 \,\pi(dx) \le \frac{1}{\lambda_n} \int_{\R^d} \frac{-\cL_0 W}{W} \phi^2 \,\pi(dx).
\label{eq:lyap1}
\end{equation}
The operator $\cL_0 = \Delta - \nabla U \cdot \nabla$ is symmetric in $L^2(\pi)$. Taking $f = \frac{\phi^2}{W}$ and $g = W$ in \eqref{eq:symmetryL0}, one has
\begin{equation}\label{eq:lyap2}
-\int_{\R^d} \frac{\phi^2}{W} (\cL_0 W) \,\pi(dx) = \int_{\R^d} \nabla\Bigl( \frac{\phi^2}{W} \Bigr) \cdot \nabla W \,\pi(dx).
\end{equation}
Direct computation yields
\[
\nabla\Bigl( \frac{\phi^2}{W} \Bigr) \cdot \nabla W = \frac{2\phi}{W} \nabla \phi \cdot \nabla W - \frac{\phi^2}{W^2} |\nabla W|^2.
\]
For any vectors $a,b$, the elementary inequality $2a \cdot b - |b|^2 \le |a|^2$ holds. Applying this with $a = \nabla \phi$ and $b = \frac{\phi}{W} \nabla W$, one obtains pointwise
\[
\frac{2\phi}{W} \nabla \phi \cdot \nabla W - \frac{\phi^2}{W^2} |\nabla W|^2 \le |\nabla \phi|^2.
\]
Hence, one has
\[
\int_{\R^d} \phi^2 \,\pi(dx) \le 
\frac{1}{\lambda_n} \int_{\R^d} \nabla\Bigl( \frac{\phi^2}{W} \Bigr) \cdot \nabla W \,\pi(dx)
\le 
\frac{1}{\lambda_n} \int_{\R^d} |\nabla \phi|^2 \,\pi(dx),
\]
which completes the proof.
\end{proof}

We now construct a sequence of asymptotic relaxation enhancing flows on $\mathbb{R}^2$.
The idea is to transplant a small-scale cellular flow from the fixed unit torus $\T^2 = [-1/2,1/2]^2$ to a large region in $\mathbb{R}^2$ via a family of diffeomorphisms $Z_n$, and to use a Lyapunov estimate to control the mass of eigenfunctions at infinity.

For each $n\ge 1$, let $\T_{2n}^2$ be the torus obtained from the square $[-2n,2n]^2$ by identifying opposite sides.
Choose a smooth cutoff function $\chi_n$ on $\T_{2n}^2$ such that $\chi_n = 1$ on $B(0, 3n/2)$ and $\chi_n = 0$ outside $B(0,7n/4)$.
For large $n$ the support of $\chi_n$ is contained in the interior of $[-2n,2n]^2$, hence $\chi_n$ defines a smooth function on $\T_{2n}^2$.
On this torus we define a probability density
\begin{equation}
\pi_n := M_n^{-1}\bigl[ (\pi - \delta_n)\chi_n + \delta_n \bigr], \qquad
\delta_n := \frac12 \inf_{z \in [-2n,2n]^2} \pi(z),
\end{equation}
where $M_n$ is the normalizing constant that makes $\int_{\T_{2n}^2}\pi_n = 1$.
By construction $\pi_n = M_n^{-1}\pi$ on $B(0,3 n/2)$, and since $\pi$ is a probability density on $\mathbb{R}^2$, we have $M_n \to 1$ as $n\to\infty$.

Let $Z_n : \T^2 \to \T_{2n}^2$ be the diffeomorphism that pushes the uniform Lebesgue measure on $\T^2$ forward to $\pi_n$.
Such a map can be built exactly as in Section~\ref{subsec:retorus}, with the origin of $\T^2$ corresponding to the centre of the construction.
Explicitly, $Z_n(x,y) = (P_n(x), Q_n(x,y))$ where
\begin{equation}\label{eq:zn}
P_n'(x) = \frac{1}{\bar{\pi}_n(P_n(x))}, \qquad
\frac{\partial Q_n}{\partial y}(x,y) = \frac{1}{\pi_n(P_n(x), Q_n(x,y))\,\bar{\pi}_n(P_n(x))},
\end{equation}
and $\bar{\pi}_n(p) = \int_{-2n}^{2n} \pi_n(p,q)\,dq$.
Because $\pi_n = M_n^{-1}\pi$ on $B(0,3n/2)$ and $\pi$ is smooth, $Z_n$ is a smooth diffeomorphism for each fixed $n$.

For an integer $m$, let $\mathcal{C}_m$ be the partition of $\T^2$ into cells of side length $1/(2m)$.
Define
\begin{equation}
D(n,m) := \bigcup_{C \in \mathcal{C}_m,\; Z_n(C) \subset B(0,\frac32 n)} Z_n(C).
\end{equation}
The cells $C$ appearing in this union are those whose images under $Z_n$ lie entirely inside the region where $\chi_n = 1$; hence these cells are unaffected by the cutoff.

We now check that for each $n$ we can choose $m$ large enough so that $D(n,m)$ contains $B(0,n)$.
Since $Z_n$ is smooth on the compact torus $\T^2$, it is Lipschitz with some constant $L_n$.
Let $K_n := Z_n^{-1}\bigl(\overline{B(0,n)}\bigr)$, a compact subset of $\T^2$.
If $m$ is so large that $\operatorname{diam}(C) = \sqrt{2}/(2m) < n/(2L_n)$ for every $C\in\mathcal{C}_m$, then for any cell $C$ that intersects $K_n$ we have
\[
Z_n(C) \subset B\bigl(Z_n(x), L_n\operatorname{diam}(C)\bigr) \subset B(0,n + \tfrac12 n) = B(0,\tfrac32 n)
\]
for some $x\in C\cap K_n$.  Hence all such cells belong to $D(n,m)$.
Their images cover $Z_n(K_n) = \overline{B(0,n)}$ and therefore contain $B(0,n)$.
Consequently, there exists an integer $m_0(n)$ such that
\[
D(n,m) \supset B(0,n) \qquad \text{for all } m \ge m_0(n).
\]

We now give the constraints on the parameter $\varphi(n)$, which will determine the cell size of our cellular flow. It is first taken to satisfy 
\begin{equation}\label{eq:cellsize}
\varphi(n) \ge m_0(n)\quad \mbox{and}\quad\varphi(n) \to \infty.
\end{equation}
In the proof of the theorem below we will further increase $\varphi(n)$ if necessary to meet an additional growth condition that involves only quantities already defined there.

Let $\Psi_{\varphi(n)}$ be the stream function of the cellular flow on $\T^2$ with cell size $1/(2\varphi(n))$, i.e.
\begin{equation}
\Psi_{\varphi(n)}(x,y) = \sin\bigl(2\uppi \varphi(n) x\bigr) \cos\bigl(2\uppi \varphi(n) y\bigr),
\end{equation}
and let $\tilde{v}_n = \nabla^\perp \Psi_{\varphi(n)}$ be the corresponding divergence-free flow with respect to the flat measure on $\T^2$.
Pushing $\tilde{v}_n$ forward by $Z_n$ yields a flow $v_n'$ on $\T_{2n}^2$ that is divergence-free with respect to $\pi_n$.
Because the pushforward preserves streamlines and the streamlines of $\tilde{v}_n$ enclose simply connected regions, there exists a function $\Psi_{\varphi(n)}'$ such that
\begin{equation}
v_n' = \frac{1}{\pi_n} \nabla^\perp \Psi_{\varphi(n)}'.
\end{equation}

Finally, we extend $v_n'$ to the whole plane $\mathbb{R}^2$ by multiplying the stream function with the cutoff $\chi_n$ and the factor $M_n$, obtaining
\begin{equation}\label{eq:vnR2}
v_n := \frac{M_n}{\pi} \nabla^\perp \bigl( \chi_n \Psi_{\varphi(n)}' \bigr).
\end{equation}
By construction, $v_n$ is smooth, compactly supported in $B(0,2n)$, and satisfies $\nabla \cdot (\pi v_n) = 0$ on $\mathbb{R}^2$.
Thus each $v_n$ has finite energy $\int_{\mathbb{R}^2} |v_n|^2 \pi(dx)$.
Moreover, on $B(0,\frac32 n)$ the flow $v_n$ coincides with $v_n'$; therefore, on $D(n,\varphi(n)) \supset B(0,n)$ the flow $v_n$ is conjugate via $Z_n$ to the cellular flow $\tilde{v}_n$ on $\T^2$.

Our main claim is the following.

\begin{theorem}
There exists a choice of $\varphi(n) \to \infty$ with $\varphi(n) \ge m_0(n)$ such that the sequence $\{v_n\}$ defined in \eqref{eq:vnR2} has finite energy for each $n$, and is asymptotic relaxation enhancing with respect to $\pi$ on $\mathbb{R}^2$.
\end{theorem}

\begin{proof}
For each $n$, let $E_n$ be the set of all eigenfunctions $\phi \in H^1(\pi)$ of $v_n \cdot \nabla$ with $\int_{\mathbb{R}^2} \phi \,\pi(dx) = 0$.
We prove that there exist constants $c_n \to \infty$ such that
\begin{equation}\label{eq:bound}
\inf_{\phi \in E_n} \frac{\int_{\mathbb{R}^2} |\nabla \phi|^2 \,\pi(dx)}{\int_{\mathbb{R}^2} |\phi|^2 \,\pi(dx)} \ge c_n.
\end{equation}

Fix $\phi \in E_n$ with $\|\phi\|_{L^2(\pi)} = 1$ and $v_n \cdot \nabla \phi = i\mu \phi$.
Set
\[
\tilde{D}_n := Z_n^{-1}\bigl(D(n,\varphi(n))\bigr), \qquad
\tilde{\phi}(x,y) := \phi\bigl(Z_n(x,y)\bigr) \text{ for } (x,y) \in \tilde{D}_n.
\]

Because $\varphi(n) \ge m_0(n)$, we have $D(n,\varphi(n)) \supset B(0,n)$, and $\tilde{D}_n$ consists of a union of whole cells.
For large $n$, $\tilde{D}_n$ is contained in a fixed compact square $K = [-1/4,1/4]^2 \subset \T^2$. On $K$, the maps $Z_n$ are smooth, hence $\|DZ_n(x)\|$ is bounded for each fixed $n$.
Define
\[
C_n := \sup_{x \in K} \|DZ_n(x)\|_{L^{\infty}}^2,
\]
which is finite for every $n$.
By the chain rule and the change of variables, we obtain for every smooth $\phi$ the inequality
\begin{equation}\label{eq:grad-ineq}
\int_{\tilde{D}_n} |\nabla \tilde{\phi}|^2 \,dx \le C_n \int_{D(n,\varphi(n))} |\nabla \phi|^2 \,\pi_n(dx).
\end{equation}

We now adjust the final choice of $\varphi(n)$.  So far $\varphi(n)$ only needs to satisfy $\varphi(n) \ge m_0(n)$ and $\varphi(n) \to \infty$.
By increasing $\varphi(n)$ further if necessary, we can also ensure that
\[
\frac{\varphi(n)^2}{C_n} \to \infty \quad \text{as } n\to\infty.
\]
This is possible because $C_n$ is already fixed for each $n$. We work with such a choice from now on.

\medskip
\noindent \textbf{Case 1: $i\mu \neq 0$.}
As in the proof of Proposition~\ref{pro:asymCell}, a nonzero eigenvalue forces $\tilde{\phi}$ to vanish on every cell of $\tilde{D}_n$.
Thus $\tilde{\phi} \equiv 0$ on $\tilde{D}_n$, and hence $\phi \equiv 0$ on $D(n,\varphi(n))$.
Since $D(n,\varphi(n)) \supset B(0,n)$, $\phi$ is supported in $\mathbb{R}^2 \setminus B(0,n)$.
Applying Lemma~\ref{lmm:poincarefaraway} to the real and imaginary parts of $\phi$ yields
\[
\int_{\mathbb{R}^2} |\nabla \phi|^2 \,\pi(dx) \ge \lambda_{n} \to \infty.
\]

\medskip
\noindent \textbf{Case 2: $i\mu = 0$.}
Now $\phi$ is a first integral; we may assume $\phi$ is real (the general case follows by applying the obtained estimate to real and imaginary parts), and $\int \phi^2 \pi = 1$.
Fix $\delta = 1/12$.

\medskip
\noindent \textit{Subcase 2.1:} $\displaystyle \int_{\tilde{D}_n} |\tilde{\phi}|^2 \,dx \ge 1 - \delta$.

By Proposition~\ref{pro:asymCell} applied cell by cell, $\tilde{\phi}$ is continuous on $\tilde{D}_n$
except possibly at cell centres, and attains a common value $a$ on the boundaries of all cells
in $\tilde{D}_n$.  For each cell $C \subset \tilde{D}_n$, the function $\tilde{\phi} - a$
vanishes on $\partial C$, so by scaling and the Poincar\'e inequality,
\[
\int_C |\nabla \tilde{\phi}|^2 \,dx \ge \kappa_0 \varphi(n)^2 \int_C |\tilde{\phi} - a|^2 \,dx,
\]
with a universal constant $\kappa_0 > 0$.  Summing over all cells in $\tilde{D}_n$ gives
\begin{equation}\label{eq:cellPoinc}
\int_{\tilde{D}_n} |\nabla \tilde{\phi}|^2 \,dx \ge \kappa_0 \varphi(n)^2 \int_{\tilde{D}_n} |\tilde{\phi} - a|^2 \,dx .
\end{equation}

To replace $a$, we use the fact that for any square-integrable $f$,
\[
\int_{\tilde{D}_n} |f - c|^2 \,dx \ge \min_{c\in\mathbb{R}} \int_{\tilde{D}_n} |f - c|^2 \,dx
= \int_{\tilde{D}_n} f^2 \,dx - \frac{1}{|\tilde{D}_n|}\Bigl(\int_{\tilde{D}_n} f \,dx\Bigr)^2 .
\]
Applying this with $f = \tilde{\phi}$ and $c = a$ yields
\[
\int_{\tilde{D}_n} |\tilde{\phi} - a|^2 \,dx
\ge \int_{\tilde{D}_n} \tilde{\phi}^2 \,dx - \frac{1}{|\tilde{D}_n|}\Bigl(\int_{\tilde{D}_n} \tilde{\phi} \,dx\Bigr)^2 .
\]

We now estimate the mean of $\tilde{\phi}$. By change of variables, we have
\[
\int_{\tilde{D}_n} \tilde{\phi} \,dx = M_n^{-1} \int_{D(n,\varphi(n))} \phi \,\pi(dx).
\]
Since $\int_{\mathbb{R}^2} \phi \,\pi(dx) = 0$, this equals $-M_n^{-1} \int_{\mathbb{R}^2\setminus D(n,\varphi(n))} \phi \,\pi(dx)$.
Applying the Cauchy-Schwarz inequality and using $\int_{\mathbb{R}^2} \phi^2\pi = 1$,
\[
\Bigl|\int_{\tilde{D}_n} \tilde{\phi} \,dx\Bigr|
\le M_n^{-1} \sqrt{\int_{\mathbb{R}^2\setminus D(n,\varphi(n))} \phi^2 \,\pi(dx)} .
\]
The hypothesis $\int_{\tilde{D}_n} |\tilde{\phi}|^2 dx = M_n^{-1} \int_{D(n,\varphi(n))} \phi^2 \pi(dx) \ge 1-\delta$ implies
\[
\int_{\mathbb{R}^2\setminus D(n,\varphi(n))} \phi^2 \,\pi(dx) \le 1 - M_n(1-\delta).
\]
Therefore,
\[
\Bigl|\int_{\tilde{D}_n} \tilde{\phi} \,dx\Bigr| \le M_n^{-1} \sqrt{1 - M_n(1-\delta)} .
\]

Next, we bound $|\tilde{D}_n|$ from below.  Because $D(n,\varphi(n)) \supset B(0,n)$,
\[
|\tilde{D}_n| = \int_{\tilde{D}_n} dx = \int_{D(n,\varphi(n))} \pi_n(dy)
= M_n^{-1} \int_{D(n,\varphi(n))} \pi(y)\,dy
\ge M_n^{-1} \int_{B(0,n)} \pi(y)\,dy .
\]
The last integral tends to $1$, so $|\tilde{D}_n| \ge \frac12$ for all large $n$, giving $\frac1{|\tilde{D}_n|} \le 2$.

Now combine these estimates.  Using $\int_{\tilde{D}_n} \tilde{\phi}^2 \ge 1-\delta$,
\begin{align*}
\int_{\tilde{D}_n} |\tilde{\phi} - a|^2 \,dx
&\ge (1-\delta) - \frac{1}{|\tilde{D}_n|}\Bigl(\int_{\tilde{D}_n} \tilde{\phi}\Bigr)^2 \\
&\ge (1-\delta) - 2 \Bigl( M_n^{-1} \sqrt{1 - M_n(1-\delta)} \Bigr)^2 .
\end{align*}
Since $M_n\to1$, the term in parentheses tends to $\sqrt{\delta}$.  Since $\delta = 1/12$,
for large $n$ we obtain
\[
\int_{\tilde{D}_n} |\tilde{\phi} - a|^2 \,dx \ge \frac14 .
\]

Insert this into \eqref{eq:cellPoinc} and then by \eqref{eq:grad-ineq} together with $\pi_n = M_n^{-1}\pi$ on $D(n,\varphi(n))$:
\[
\int_{D(n,\varphi(n))} |\nabla \phi|^2 \,\pi(dx) \ge \frac{\kappa_0}{4C_n} \, \varphi(n)^2 M_n .
\]
By the choice of $\varphi(n)$ (which guarantees $\varphi(n)^2/C_n \to \infty$) and $M_n\to1$, the right-hand side tends to infinity.

\medskip
\noindent \textit{Subcase 2.2:} $\displaystyle \int_{\tilde{D}_n} |\tilde{\phi}|^2 \,dx \le 1 - \delta$.

The pushforward measure relation gives
\[
\int_{D_{n,\varphi(n)}} |\phi|^2 \,\pi_n(dx) = \int_{\tilde{D}_n} |\tilde{\phi}|^2 \,dx \le 1 - \delta.
\]

Recall that $D_{n,\varphi(n)}$ contains $B(0,n)$. 
Since $\int_{\mathbb{R}^2} \phi^2 \pi = 1$, it follows that
\[
\int_{\mathbb{R}^2 \setminus B(0,n)} |\phi|^2 \,\pi(dx) \ge 1 - M_n(1-\delta).
\]
Because $M_n \to 1$, for all sufficiently large $n$ we have $M_n \le 1+\delta$, and therefore
\[
\int_{\mathbb{R}^2 \setminus B(0,n)} |\phi|^2 \,\pi(dx) \ge 1 - (1+\delta)(1-\delta) = \delta^2 > 0.
\]

Let $g_n$ be a smooth cutoff function with $g_n = 1$ on $B(0, n-1)$, $g_n = 0$ outside $B(0, n)$, and $|\nabla g_n| \le C$. Set $\phi^c = \phi(1 - g_n)$, which is supported outside $B(0, n-1)$. By Lemma~\ref{lmm:poincarefaraway},
\[
\int_{\mathbb{R}^2} |\nabla \phi^c|^2 \,\pi(dx) \ge \lambda_{n-1} \int_{\mathbb{R}^2} |\phi^c|^2 \,\pi(dx) \ge \lambda_{n-1} \delta^2.
\]

Using $|\nabla \phi|^2 \ge \frac{1}{2} |\nabla \phi^c|^2 - |\phi|^2 |\nabla g_n|^2$ and $\int |\phi|^2 \pi = 1$, we get
\[
\int_{\mathbb{R}^2} |\nabla \phi|^2 \,\pi(dx) \ge \frac{1}{2} \delta^2 \, \lambda_{ n-1 } - C^2.
\]
The right hand side diverges.

\medskip
In all cases, we have produced a sequence $c_n \to \infty$ such that $\int |\nabla \phi|^2 \pi \ge c_n$ for every normalized $\phi \in E_n$. This proves \eqref{eq:bound}, so $r(v_n) \to \infty$ and $\{v_n\}$ is asymptotic relaxation enhancing.
\end{proof}

\section*{Acknowledgments}
This work was financially supported by the National Key R\&D Program of China, Project No. 2025YFA1018300, 2021YFA1002800 and 2021YFA1001200. The work of Y. Feng was partially supported by NSFC 12301283, Science and Technology Commission of Shanghai Municipality (No. 22DZ2229014). The work of L. Li was partially supported by NSFC 12371400. The work of X. Xu was partially supported by Kunshan Shuangchuang Talent Program kssc202102066.

\bibliographystyle{plain}
\bibliography{Enhanced-dissipation}

\end{document}